\documentclass[a4paper,12pt]{article}

\usepackage{pstricks}
\usepackage{graphicx,psfrag}
\usepackage{amsmath}
\usepackage{amssymb}
\usepackage{amsthm}

      \newcommand {\al}   {\alpha}          
      \newcommand {\gam } {\gamma}          
      \newcommand {\del}  {\delta}          \newcommand {\Del} {\Delta}
              \newcommand {\ve}   {\varepsilon}
                 \newcommand {\vphi} {\varphi}
      \newcommand {\lam}  {\lambda}         
                \newcommand {\Up}   {\Upsilon}
                \newcommand {\Om}  {\Omega}
      \newcommand {\pl}   {\partial}        \newcommand {\s}    {\sigma}

      \newcommand {\HHH}    {\mathcal{H}}
            
      \newcommand {\RR}  {{\mathbb R}}           
      
      \newcommand {\OOO}  {{\cal O}}        \newcommand {\MMM}  {{\cal M}}
           
              \newcommand {\PPP}  {{\cal P}}
      \newcommand {\FFF}  {{\cal F}}        
           \newcommand {\NNN}  {\mathcal{N}}
       \newcommand {\ttt}  {\varsigma}

     \newcommand {\beq}  {\begin{equation}}
      \newcommand {\eeq}  {\end{equation}}  \newcommand {\lab}  {\label}
     \newcommand {\beqo}  {\begin{equation*}}
      \newcommand {\eeqo}  {\end{equation*}}

      \newtheorem{theorem}{Theorem}
      \newtheorem{lemma}{Lemma}
      
      \newtheorem{zam}{Remark}
      \newtheorem{opr}{Definition}
      \newtheorem{corollary}{Corollary}

\author{Alexander Plakhov\thanks{Center for R\&{}D in Mathematics and Applications, Department of Mathematics, University of Aveiro, Portugal and Institute for Information Transmission Problems, Moscow, Russia.}}

\title{Local properties of the surface measure\\ of convex bodies}

\begin{document}

\maketitle

\begin{abstract}
It is well known that any measure in $S^2$ satisfying certain simple conditions is the surface measure of a bounded convex body in $\RR^3$. It is also known that a local perturbation of the surface measure may lead to a nonlocal perturbation of the corresponding convex body. We prove that, under mild conditions on the convex body, there are families of perturbations of its surface measure forming line segments, with the original measure at the midpoint, corresponding to {\it local} perturbations of the body. Moreover, there is, in a sense, a huge amount of such families. We apply this result to Newton's problem of minimal resistance for convex bodies.
\end{abstract}

\begin{quote}

\end{quote}

\begin{quote} {\small {\bf Key words and phrases:}
Convex sets, Blaschke addition, Newton's problems of minimal resistance.}
\end{quote}


\section{Motivation of the study}\label{1.1}
The motivation for this study came from the famous Newton problem of minimal resistance. In modern terms and in a generalized form the problem can be formulated as follows.

Let $C_1 \subset C_2 \subset \RR^3$ be convex compact bodies. Denote by $\mathfrak{C}(C_1,C_2)$ the set of convex bodies $C$ satisfying $C_1 \subset C \subset C_2$. Let $f : S^2 \to \RR$ be a continuous function. The {\it generalized Newton problem} is as follows:
\beq\label{probl C}
\inf_{C \in \mathfrak{C}(C_1,C_2)} F(C), \qquad \text{where} \quad F(C) = \int_{\pl C} f(n_\xi)\, d\HHH^2(\xi).
\eeq
Here and in what follows, $\HHH^2$ designates the 2-dimensional Hausdorff measure in $\RR^3$, and $n_\xi$ is the outer normal to $C$ at a regular point $\xi \in \pl C$.

In this form the problem first appears in the paper by Buttazzo and Guasoni \cite{BG97}. In fact, they consider an even more general case when $f$ depends not only on $n_\xi$, but also on $\xi$, and investigate some other classes of convex bodies, along with $ \mathfrak{C}(C_1,C_2)$.

Of special interest is the particular case when $C_2$ is a right circular cylinder and $C_1$ is its rear base. In an appropriate orthogonal coordinate system with the coordinates $x,\, y,\, z$,\, $C_1$ and $C_2$ take the form: $C_1 = \Om \times \{ 0 \}$ and $C_2 = \Om \times [0,\, M]$, with $\Om = \{ (x,y) : x^2 + y^2 \le 1 \}$ and $M > 0$. Additionally, the function $f$ is taken to be $f(n) = \langle n,\, e_3 \rangle_+^3$, where $e_3 = (0,0,1)$ and $b_+$ is the positive part of $b \in \RR$. Here and in what follows, $\langle \cdot\,, \cdot \rangle$ means the scalar product.

In this particular case, a body $C \in \mathfrak{C}(C_1,C_2)$ is the subgraph of a concave function $u_C : \Om \to \RR$ satisfying $0 \le u_C(x,y) \le M$, namely, $C = \{ (x,y,z) : (x,y) \in \Om,\ 0 \le z \le u_C(x,y) \}$. Further, $\pl C$ is the union of the disc $\Om \times \{ 0 \}$ and the graph of $u_C$; for $\xi = (x,y,0) \in \Om \times \{ 0 \}$ one has $n_\xi = (0,0,-1)$ and $f(n_\xi) = 0$, and for $\xi = (x, y, u_C(x,y))$ one has
$n_\xi = (1 + |\nabla u_C(x,y)|^2)^{-1/2}\, ( -\frac{\pl u_C}{\pl x}(x,y),\, -\frac{\pl u_C}{\pl y}(x,y),\, 1 )$,\,
$f(n_\xi) = (1 + |\nabla u_C(x,y)|^2)^{-3/2}$, and $d\HHH^2(\xi) = (1 + |\nabla u_C(x,y)|^2)^{1/2}\, dx\, dy$. In view of this, Problem \eqref{probl C} can be written as follows: minimize the integral
\beq\label{probl u}
\int\!\!\!\int_\Om \frac{1}{1 + |\nabla u(x,y)|^2}\, dx\, dy
\eeq
in the class of concave functions $u : \Om \to \RR$ satisfying the condition $0 \le u(x,y) \le M$. In what follows, Problem \eqref{probl u} will be called the {\it particular Newton problem}.

The particular problem was first stated by Buttazzo and Kawohl in \cite{BK} and has been intensively studied since then; see, e.g., \cite{BFK,BrFK,CL1,CL2,LP1,LP2,LO}. However, it has not been completely solved yet. It is proved that the solution $u$ exists \cite{BFK} and has the following properties: first, if $\nabla u$ exists at a certain point $(x,y) \in \Om$ then either $|\nabla u(x,y)|  = 0$ or $|\nabla u(x,y)|  \ge 1$ \cite{BFK}; second, if $u$ is in the class $C^2$ in a neighborhood of $(x,y)$ and $0 < u(x,y) < M$ then the matrix of the second derivative $D^2 u(x,y)$ has a zero eigenvalue \cite{BrFK}. Properties of the solution were numerically examined in \cite{LO,W}. Note that problem \eqref{probl u} in the subclass of concave radially symmetric functions was first considered by I. Newton in his {\it Principia} (1687).

\section{The surface measure of convex bodies}\label{1.2}

Let $C \subset \RR^3$ be a convex compact body. By $|A|$ or $\HHH^2(A)$ we denote the 2-dimensional Hausdorff measure of a Borel set $A \subset \pl C$.

Define the map $\NNN_C : \pl'C \to S^2$ by $\NNN_C(\xi) = n_\xi$, where $\pl'C$ denotes the (full-measure) subset of regular points of $\pl C$. The pushforward measure $\nu_C := \NNN_C\# \HHH^2$ defined in $S^2$ is called the {\it surface measure of $C$}. Thus, for a Borel set $A \subset S^2$ we have $\nu_C(A) = |\NNN_C^{-1}(A)|$.

The surface measure $\nu = \nu_C$ of a convex compact body $C$ satisfies the following conditions.

(a) $\int_{S^2} n\, d\nu(n) = \vec 0$.

(b) $\nu(S^2 \setminus \Pi) > 0$ for any 2-dimensional subspace $\Pi$ of $\RR^3$.\\
By Alexandrov's theorem \cite{Alex}, conversely, if a measure $\nu$ defined in $S^2$ satisfies (a) and (b) then there exists a unique (up to a translation) convex compact body $C$ with the surface measure equal to $\nu$, that is, $\nu_C = \nu$.

The notion of surface measure allows one to define the sum of two convex compact bodies (Blaschke sum); namely, the sum of $A$  and $B$ is (up to a translation) the convex compact body $A \# B$ such that $\nu_{A\# B} = \nu_A + \nu_B$.

In terms of surface measure, Problem \eqref{probl C} can be rewritten as a  minimization problem for a linear functional on a space of measures,
\beq\label{probl linear}
\inf_{\nu \in \Up(C_1,C_2)} \FFF(\nu), \qquad \text{where} \quad \FFF(\nu) = \int_{S^2} f(n)\, d\nu(n),
\eeq
where $\Up(C_1,C_2)$ designates the set of surface measures $\nu_C$ with $C_1 \subset C \subset C_2$; that is, $\Up(C_1,C_2) := \{ \nu_C : C_1 \subset C \subset C_2 \}$.


The fact that Newton's problem can be formulated as a linear problem in terms of surface measure was first noticed by Carlier and Lachand-Robert in \cite{CLR03}. They used the linear representation to study the generalized Newton problem (i) in the class of convex bodies with fixed surface area and (ii) in the more restricted class of bodies with fixed surface area and, additionally, with the fixed area of the projection of the body on a certain plane $e_1 x + e_2 y + e_3 z = 0$.

These classes admit an easy and natural representation in terms of the surface measure. Indeed, the class of surface measures corresponding to (i) contains all measures $\nu$ satisfying (a) and (b) with fixed full measure $\nu(S^2)$, while  the more restricted class of surface measures corresponding to (ii) contains the measures satisfying the additional condition that the measure of the hemisphere $\{ n \in S^2 : \langle n,\, e \rangle \ge 0 \}$ is fixed.

 On the contrary, the main difficulty of Problem \eqref{probl linear} resides in the extreme complexity of the set of measures $\Up(C_1,C_2)$. Nevertheless, we believe that studying this set could trigger further progress in Newton's problem and additionally, extend our knowledge about the surface measure and Blaschke addition.

Here are some intriguing questions concerning $\Up(C_1,C_2)$, with reasonable sets $C_1$ and $C_2$ (for example, a cylinder and one of its bases, or two concentric balls).

$\bullet$ Characterize the convex hull, Conv$\Up(C_1,C_2)$, of the set $\Up(C_1,C_2)$ and the class of convex bodies with the surface measure contained in Conv$\Up(C_1,C_2)$.

$\bullet$ Characterize the set $\Up(C_1,C_2) \cap \pl\, (\text{Conv} \Up(C_1,C_2))$. Each measure contained in it is a solution to a generalized Newton problem \eqref{probl linear}.

$\bullet$ Characterize the extreme points of Conv$\Up(C_1,C_2)$. Any such point is a unique solution to a generalized Newton problem \eqref{probl linear}.

It would be interesting to develop a computer algorithm of solving Problem \eqref{probl linear}. In this connection we mention the papers \cite{Kuta,Z}, where numerical algorithms for representation of Blaschke sum of convex polyhedra are proposed.

\section{Main results}\label{1.3}

Here we state the main results of the paper concerning local properties of surface measures, and deduce from them some properties of a solution to a generalized Newton problem.

Let $C$ be a convex compact body in $\RR^3$. Take a point $O$ outside $C$ and consider the closed cone $K$ with the vertex at $O$ circumscribed about $C$; in other words, $K$ is the union of all rays intersecting $C$ with the origin at $O$. Clearly, $K$ is convex. Denote by $\pl_- C$ the part of the boundary of $C$ contained in the interior of the convex hull of $C \cup \{O\}$,
$$
\pl_- C := \pl C \cap \text{\,int}\,(\text{Conv}(C \cup \{O\})),
$$
and let $\pl_+ C$ be the complementary part of the boundary,
$$
\pl_+ C := \pl C \setminus \pl_- C = \pl C \cap \pl (\text{Conv}(C \cup \{O\}).
$$
Draw a line $l$ through $O$ intersecting the interior of $C$. Let $B$ and $B'$ be the points of intersection of $l$ with $\pl C$, where the semiopen segment $[O,\, B)$ lies outside $C$, and the closed segment $[B,\, B']$ is contained in $C$; see Fig.~\ref{fig body}. The  intersection of a half-plane bounded by $l$ with $\pl C$ is a planar convex curve with the endpoints $B$ and $B'$ (the curve $BMB'$ in Fig.~\ref{fig body}), and the intersection of this half-plane with $\pl K$ is the ray of support to this curve with the origin at $O$ (the ray $OM$ in Fig.~\ref{fig body}).

   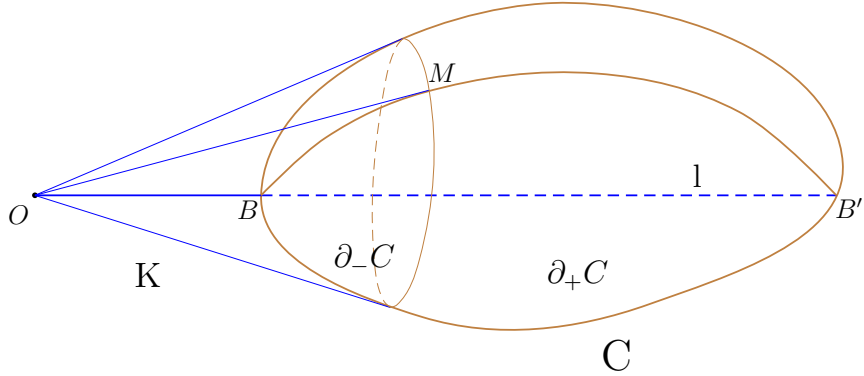
\begin{figure}[h]
\begin{picture}(0,160)
\scalebox{0.85}{
\rput(4,0.25){
\psecurve[linecolor=brown](0.5,2)(3,0.1)(6.5,0.3)(9.5,2.5)(5,5)(0.5,2)(3,0.1)(6.5,0.3)
        \psdots[dotsize=2.5pt](-3,2)   
\psline[linecolor=blue](-3,2)(0.5,2)
\psline[linecolor=blue,linestyle=dashed](0.5,2)(9.4,2)
\pscurve[linecolor=brown](0.5,2)(1.5,2.9)(3,3.6)(5.7,3.9)(7.9,3.3)(9.4,2)
                 \rput(3.3,3.92){$M$}
\psline[linecolor=blue,linewidth=0.4pt](-3,2)(2.7,4.45)
\psline[linecolor=blue,linewidth=0.4pt](-3,2)(3.1,3.64)
       \psline[linecolor=blue,linewidth=0.4pt](-3,2)(2.5,0.25)
\pscurve[linecolor=brown,linewidth=0.4pt](2.7,4.45)(2.8,4.4)(3.1,3.64)
(2.95,1)(2.6,0.29)(2.5,0.25)
\pscurve[linecolor=brown,linewidth=0.4pt,linestyle=dashed](2.7,4.45)(2.6,4.33)(2.4,3.64)
   (2.25,1)(2.4,0.35)(2.5,0.25)
\rput(-3.25,1.7){$O$}
\rput(9.6,1.8){$B'$}
\rput(0.3,1.77){$B$}
       \rput(2.1,1){\scalebox{1.2}{$\pl_- C$}}
       \rput(5.4,0.75){\scalebox{1.2}{$\pl_+ C$}}
       \rput(7.25,2.3){$\scalebox{1.2}{l}$}
       \rput(-1.25,0.75){$\scalebox{1.3}{K}$}
              \rput(6,-0.5){$\scalebox{1.6}{C}$}
}}
\end{picture}
\caption{The convex body $C$ and the circumscribed cone $K$.}
\label{fig body}
\end{figure} 

\begin{theorem}\label{t0}
Assume that

(i) in each half-plane, the corresponding ray of support touches the curve;

(ii) there exists $k > 0$ such that in any half-plane, the part of the curve between $B$ and the point of tangency is of class $C^1$ with the curvature $\le k$.

Then there exists a family of convex bodies $C(s),\ s \in [-1,\, 1]$ such that

(a) $C(0) = C$;

(b) $\pl_+ C \subset \pl C(s)$;

(c) the corresponding family of surface measures $\nu_{C(s)}$ is a line segment with the midpoint at $\nu_C$; that is, for all $s \in [-1,\, 1]$
$$
\nu_{C(s)} = \nu_{C} + s\, (\nu_{C(1)} - \nu_C).
$$
In particular, $\nu_C = \frac{1}{2} \big( \nu_{C(-1)} + \nu_{C(1)} \big)$, and the signed measure $\nu_{C(1)} - \nu_{C}$ is a director vector of the segment.

(d) Moreover, the aforementioned family of bodies is not unique; there exists a set of families of bodies $C^{(\theta)}(s)$ depending on the parameter $\theta \in \Theta$, with each family satisfying (a), (b), (c), such that the union of 1-dimensional subspaces spanned by the corresponding director vectors, $\cup_{\theta \in \Theta} \{ \lam(\nu_{C^{(\theta)}(1)} - \nu_C) : \lam \in \RR \}$, is an infinite-dimensional subspace of the space of signed measures.
\end{theorem}

\begin{zam}\label{z curvature}
By saying that the curvature of a $C^1$ curve is $\le k$ we mean that, given a natural parametrization $\gam(\ttt)$ of the curve, the angle between any two unit vectors $\gam'(\ttt_1)$ and $\gam'(\ttt_2)$ does not exceed $k|\ttt_2 - \ttt_1|$.   
\end{zam}

The proof of Theorem \ref{t0} will be given in Section \ref{sec proof}.

Note that a local perturbation of the surface $\pl C$ of a convex body $C$ leads to a local perturbation of its surface measure $\nu_C$. The converse, however, is not true: it is well known that a local perturbation of the surface measure may lead to a nonlocal perturbation of the body surface. In other words, if we have a domain $\Om \subset S^2$ and two surface measures $\nu_{C_1}$ and $\nu_{C_2}$ such that $\nu_{C_1}\rfloor_\Om = \nu_{C_2}\rfloor_\Om$, it does not follow that $\NNN_{C_1}^{-1}(\Om)$ and $\NNN_{C_2}^{-1}(\Om)$ coincide up to a translation.

 Indeed, let $e_1 = (1,0,0)$,\, $e_2 = (0,1,0)$,\, $e_3 = (0,0,1)$, take $\al \ne 0$ and denote $v^+ = {(e_3 + \al e_1)}/{\sqrt{1 + \al^2}}$,\, $v^- = {(e_3 - \al e_1)}/{\sqrt{1 + \al^2}}$. One obviously has $e_3 = \sqrt{1 + \al^2}\, (v^+ + v^-)/2$. The measures
 $$
 \nu_{C_1} = \del_{e_1} + \del_{-e_1} + \del_{e_2} + \del_{-e_2} + \del_{e_3} + \del_{-e_3}
 $$
 and
 $$
 \nu_{C_2} = \del_{e_1} + \del_{-e_1} + \del_{e_2} + \del_{-e_2} +  \frac{\sqrt{1 + \al^2}}{2}\, (\del_{v^+} + \del{v^-}) + \del_{-e_3}
 $$
 are surface measures of a cube and a body called "house"; see Fig.~\ref{fig cube and house}.  Let $\Om \subset S^2$ be any domain that does not include $e_3,\, v^+$, and $v^-$. One has $\nu_{C_1}\rfloor_\Om = \nu_{C_2}\rfloor_\Om$; however, the sets $\NNN_{C_1}^{-1}(\Om)$ and $\NNN_{C_2}^{-1}(\Om)$ are not translations of each other. Indeed, $\NNN_{C_1}^{-1}(\Om)$ is the union of the lateral and the lower faces of the cube, whereas $\NNN_{C_2}^{-1}(\Om)$ is the union of the "walls" and the "floor" of the "house".

  \begin{figure}[h]
\begin{picture}(0,90)
\scalebox{0.8}{
\rput(1.75,0){
\pspolygon(0,0)(3,0)(3,3)(0,3)
\psline(3,0)(4.5,0.5)(4.5,3.5)(1.5,3.5)(0,3)
\psline(3,3)(4.5,3.5)
\psline[linewidth=0.3pt,linestyle=dashed](0,0)(1.5,0.5)(1.5,3.5)
\psline[linewidth=0.3pt,linestyle=dashed](1.5,0.5)(4.5,0.5)
}
\rput(9.25,0){
\pspolygon(0,0)(3.5,0)(3.5,2.5)(0,2.5)
\psline(3.5,0)(5,0.5)(5,3)(4,3.5)(0.5,3.5)(0,2.5)
\psline(3.5,2.5)(4,3.5)
\psline[linewidth=0.3pt,linestyle=dashed](0,0)(1.5,0.5)(1.5,3)(0.5,3.5)
\psline[linewidth=0.3pt,linestyle=dashed](1.5,0.5)(5,0.5)
\psline[linewidth=0.3pt,linestyle=dashed](1.5,3)(5,3)
}
}
\end{picture}
\caption{A cube and a "house".}
\label{fig cube and house}
\end{figure}
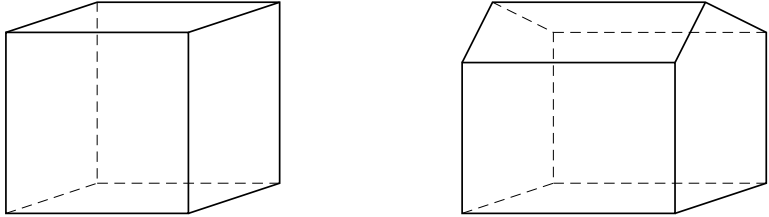 

It is important therefore to find out conditions which would guarantee that a local perturbation of the measure implies a local perturbation of the body surface.

Theorem \ref{t0} is a step in this direction; it states that, under mild conditions, there is a family of perturbations of the measure forming a line segment, with the original measure at the midpoint, leading to a local perturbation of the body surface. Moreover, there is, in a sense, a huge amount of such perturbations.

Theorem \ref{t0} can also be interpreted as follows. Consider all possible convex bodies obtained from $C$ by perturbations of its boundary in the subset $\pl_- C$, and consider the set of their surface measures. Then $\nu_C$ is not an extreme point of this set.

\begin{corollary}\label{cor1}
Assume that an open subset $\OOO \subset \pl C$ of the body's boundary is of class $C^2$ and its gaussian curvature is positive. Take a point $O \not\in C$ so as the closure of the set $\pl_- C = \pl C \cap \text{\rm int}\,(\text{\rm Conv}(C \cup \{O\}))$ is contained in $\OOO$, and therefore, $\pl_+ C = \pl C \setminus \pl_- C$ contains $\pl C \setminus \OOO$. Then the statement of Theorem \ref{t0} holds true.
\end{corollary}

\begin{proof}
Draw a line $l$ through $O$ intersecting the interior of $C$, denote by $B$ and $B'$ the points of intersection  of $l$ with $\pl C$, as explained in the beginning of this subsection, and consider all half-planes bounded by $l$. Note that all rays of support to $C$ from $O$ intersect $\pl C$ inside $\OOO$, and therefore, are tangent to $C$. Thus, condition (i) of the theorem is satisfied.

Consider the concave curves resulting from the intersection of the half-planes with $\pl C$. The parts of these curves between $B$ and the point of tangency of the corresponding ray of support lie in $\overline{\pl_- C}$, and therefore are of class $C^2$. Their curvature is a continuous function defined on the compact set with the coordinates ({\it the distance of the point from $l$, the angle of inclination of the corresponding half-plane}), and therefore does not exceed a positive constant $k$. Thus, condition (ii) of Theorem \ref{t0} is also valid, and therefore, the statement of the theorem holds true.
\end{proof}

\begin{zam}\label{z conditions}
The hypothesis of Theorem \ref{t0} is strictly weaker than the hypothesis of Corollary \ref{cor1}. Indeed, let $C$ be the convex hull of two planar domains, $C = \rm{Conv}(D_1 \cup D_2),$ where $D_1 = \{ (x,y,z) : |x| \le 1,\ |z| \le 1 - x^2, \ y = 0 \}$ and $D_2 = \{ (x,y,z) : |x| \le 1,\ |y| \le 1 - x^2, \ z = 0 \}$; see Fig.~\ref{fig zam}. Then the hypothesis of Theorem \ref{t0} holds true for $O = (c,0,0)$ with $c > 1$,\, $B = (1,0,0)$,\, $B' = (-1,0,0)$,\,  and $k = 2$; however, the gaussian curvature is zero at any regular point of $\pl C.$
\end{zam}
    
     \begin{figure}[h]
\begin{picture}(0,110)
\scalebox{1}{
\rput(5.5,2){
\pscurve[linecolor=brown](-2,0)(-1.5,0.875)(-1,1.5)(-0.5,1.875)(0,2)(0.5,1.875)(1,1.5)(1.5,0.875)(2,0)
\pscurve[linecolor=brown](-2,0)(-1.2375,0.4375)(-0.55,0.75)(0.0625,0.9375)(0.6,1)(1.0625,0.9375)(1.45,0.75)(1.7625,0.4375)(2,0)
\pscurve[linecolor=brown](-2,0)(-1.5,-0.875)(-1,-1.5)(-0.5,-1.875)(0,-2)(0.5,-1.875)(1,-1.5)(1.5,-0.875)(2,0)
\pscurve[linewidth=0.3pt,linestyle=dashed,linecolor=brown](2,0)(1.2375,-0.4375)(0.55,-0.75)(-0.0625,-0.9375)(-0.6,-1)(-1.0625,-0.9375)(-1.45,-0.75)(-1.7625,-0.4375)(-2,0)
\psline[linewidth=0.2pt,linecolor=brown](-1.5,0.875)(-1.2375,0.4375)(-1.5,-0.875)
\psline[linewidth=0.2pt,linecolor=brown](-1,1.5)(-0.55,0.75)(-1,-1.5)
\psline[linewidth=0.2pt,linecolor=brown](-0.5,1.875)(0.0625,0.9375)(-0.5,-1.875)
\psline[linewidth=0.2pt,linecolor=brown](0,2)(0.6,1)(0,-2)
\psline[linewidth=0.2pt,linecolor=brown](0.5,1.875)(1.0625,0.9375)(0.5,-1.875)
\psline[linewidth=0.2pt,linecolor=brown](1,1.5)(1.45,0.75)(1,-1.5)
\psline[linewidth=0.2pt,linecolor=brown](1.5,0.875)(1.7625,0.4375)(1.5,-0.875)
\psline[linestyle=dashed,linewidth=0.4pt,linecolor=blue](-2,0)(2,0)
\psline[linewidth=0.8pt,linecolor=blue](4,0)(2,0)
\rput(-2.3,-0.1){$B'$}
\rput(2.15,-0.25){$B$}
\rput(4.1,-0.25){$O$}
\psdots[dotsize=2.5pt](4,0)
}}
\end{picture}
\caption{A convex body satisfying the hypothesis of Theorem \ref{t0}.}
\label{fig zam}
\end{figure}
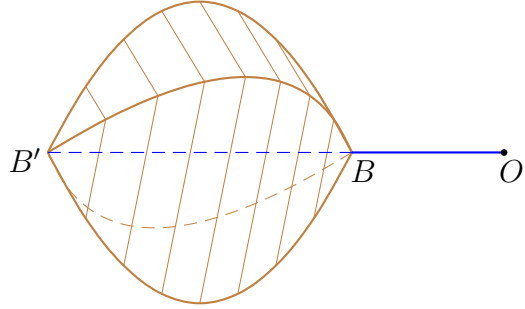 

From Theorem \ref{t0} one derives an important statement concerning linear Problem \eqref{probl linear}, and therefore also the corresponding generalized Newton problem \eqref{probl C}.

\begin{theorem}\label{t_linear}
Let $C_1 \subset C \subset C_2$, and let the point $O \not\in C$ be such that the set $\pl_- C = \pl C \cap \text{\rm int}\,(\text{\rm Conv}(C \cup \{O\}))$ does not intersect $\pl C_1$ and $\pl C_2$, that is, $\pl_- C \cap (\pl C_1 \cup \pl C_2) = \emptyset$. Draw a line $l$ through $O$ and the interior of $C$, and assume that the hypothesis of Theorem \ref{t0} is satisfied. Then either $\nu_C$ is not a solution to Problem \eqref{probl linear}, or $\nu_C$ is included in a huge set of solutions: there exists an infinite-dimensional linear space $\MMM$ of signed measures such that for any $\nu \in \MMM$ and some $\ve(\nu) > 0$, all measures of the kind $\nu_C + s \nu, \ s \in [-\ve(\nu),\, \ve(\nu)]$, are solutions.
\end{theorem}

\begin{proof}
Indeed, making if necessary the point $O$ sufficiently close to $C$ on the line $l$, one can assure that, on one hand, the set {Conv}$\,(C \cup \{O\})$ is contained in $C_2$ and, on the other hand, {Conv}$\,(\pl_+ C)$ contains $C_1$. All convex sets, corresponding to the elements of the segment of measures indicated in the statement of Theorem \ref{t0}, contain $\pl_+ C$, and therefore, also contain $C_1$. Further, all such convex sets lie in the intersection of all half-spaces bounded by planes of support through points of $\pl_+ C$, and this intersection of half-spaces coincides with {Conv}$\,(C \cup \{O\})$. This proves that these convex sets are contained in $C_2$.

If $\nu_C$ is a solution to Problem \eqref{probl linear} then for any $\theta \in \Theta$, all measures $\nu_{C^{(\theta)}(s)} = \nu_{C} + s\, (\nu_{C^{(\theta)}(1)} - \nu_C), \ s \in [-1,\, 1]$ forming the corresponding linear segment are also solutions. Take $\MMM = \cup_{\theta \in \Theta} \{ \lam(\nu_{C^{(\theta)}(1)} - \nu_C) : \lam \in \RR \}$; by statement (d) of Theorem \ref{t0}, $\MMM$ is an infinite-dimensional space. Further, for any $\nu \in \MMM$ there exist $\lam \in \RR$ and $\theta \in \Theta$ such that $\nu = \lam(\nu_{C^{(\theta)}(1)} - \nu_C)$, and taking $\ve(\nu) = 1/|\lam|$, one obtains the statement of Theorem \ref{t_linear}.
\end{proof}

Going back to the particular Newton problem \eqref{probl u}, we first note that the set $\pl C \setminus (\pl C_1 \cup \pl C_2)$ coincides with the part of the graph of $u$ with $(x,y) \in \text{int}(\Om)$ and $0 < u(x,y) < M$, where int$(\Om)$ means the interior of $\Om$. Recall that to each $u$ one naturally assigns the surface measure of the corresponding convex body $\{ (x,y,z) : (x,y) \in \Om,\ 0 \le z \le u(x,y) \}$.

Take a point $(x_0, y_0) \in \text{int}(\Om)$ and a value $z_0 > u(x_0, y_0)$. For any $\vphi$ consider the ray of support $z = z_0 + \kappa_\vphi t, \ t \ge 0$ in the half-plane $x = x_0 + t\cos\vphi,\, y =y_0 + t\sin\vphi,\, t \ge 0$  to the graph of the induced function of one variable $z = u(x_0 + t\cos\vphi,\, y_0 + t\sin\vphi),\ t \ge 0$, and fix the smallest value $t = t_\vphi > 0$ corresponding to the intersection of the ray of support with the graph.

Applying Theorem \ref{t_linear} to Problem \eqref{probl u}, we come to the following statement.

\begin{corollary}\lab{cor2}
Let $(x_0,y_0) \in \text{\rm int}(\Om)$ and $0 < u(x_0,y_0) < z_0 < M$. Suppose that (i) for all $\vphi$, the corresponding ray of support touches the graph of the corresponding induced function, and (ii) there exists $k > 0$ such that $\frac{d^2}{dt^2}\, u(x_0 + t\cos\vphi,\, y_0 + t\sin\vphi) \le k$ for all $\vphi$ and $0 \le t \le t_\vphi$. Then either $u$ is not a solution to  \eqref{probl u}, or $u$ is included in a huge set of solutions: let $\nu$ be the measure corresponding to $u$; there exists an infinite-dimensional linear space $\MMM$ of signed measures in $S^2$ such that for any $\nu' \in \MMM$ and some $\ve(\nu') > 0$, all measures of the kind $\nu + s\nu', \ s \in [-\ve(\nu'),\, \ve(\nu')]$, correspond to solutions of  \eqref{probl u}.
\end{corollary}

Now using Corollaries \ref{cor1} and \ref{cor2}, one comes to

\begin{corollary}\label{cor3}
Assume that in an open subset of $\Om$ the function $u$ is of class $C^2$ and the matrix of second derivatives $D^2 u$ is non-degenerate (and therefore positive definite). Then the statement of Corollary \ref{cor2}  holds true.
\end{corollary}

\begin{zam}\label{zf}
Corollaries \ref{cor2} and \ref{cor3} remain true if the integrand $1/(1 + |\nabla u(x,y)|^2)$ in \eqref{cor2} is substituted with $f(\nabla u(x,y))$, where $f : \RR^2 \to \RR$ is an arbitrary bounded continuous function.
\end{zam}

Compare Corollaries \ref{cor2} and \ref{cor3} with the result obtained in \cite{BrFK} (Remark 3.4): if in an open subset of $\Om$ the function $u$ is of class $C^2$ and the matrix of second derivatives $D^2 u$ is non-degenerate, then $u$ is not a solution to the particular Newton problem. Apart from an obvious similarity of these statements, there is some difference. Namely, according to Remark \ref{zf}, our statements remain valid in a more general case when the integrand in \eqref{probl u} is substituted with $f(\nabla u(x,y))$. On the other hand, we do not claim that $u$ is not a solution, but instead, if $u$ is a solution then the set of solutions is extremely large and $u$ is not its extreme point.

\section{Proof of the main theorem}\label{sec proof}

We will prove a slightly stronger version of Theorem \ref{t0}, with condition (ii) replaced by a weaker condition (ii)$'$; see below. Condition (i) of the theorem remains unchanged.

The idea of the proof is as follows. We take a 1-parameter family of cones circumscribed about the body $C$. The vertices of these cones form a line segment outside $C$. Then we translate each of these cones along the segment, with the translation continuously depending on the parameter. Under certain conditions, the family of translated cones is tangent to another convex body. Knowing the magnitude of translation and the surface measure of the original body $C$, one is able to determine the surface measure of the new body. Using the obtained description of the correspondence of measures, one constructs a family of translations of the family of cones that defines a family of convex bodies with their surface measures forming a line segment, with $\nu_C$ at the midpoint. This construction depends on a functional parameter.

Introduce some additional notation. Determine an orthogonal coordinate system with the coordinates $x,\, y,\, z$ so as the point $O$ is at the origin and the positive $x$-semiaxis coincides with the ray $OB$. Consider a $C^1$ function $\al(t),$ $t \in [0,\, 1]$ such that $\al(0) = 0,\ \al'(t) > 0$ for all $0 \le t \le 1$ and $(\al(1), 0, 0)$ coincides with the point $B$. Denote $B_t = (\al(t), 0, 0)$; thus, $B_0 = O$ and $B_1 = B.$

Let $K_t$ be the closed cone with the vertex at $B_t$ circumscribed about $C$. Clearly, all cones $K_t$ are convex, and $K_0 = K$. If $B$ is a regular point of $\pl C$, then $K_1$ is a half-space bounded by the tangent plane to $C$ through $B$.

Using the circumscribed cones, the body $C$ can be represented as
$$
C = \big( \cap_{0 \le t \le 1} K_t \big) \cap \text{Conv} \big( C,\, \{O\} \big).
$$

For any $\vphi$ designate by $\Pi_\vphi$ the half-plane $y = r\cos\vphi,\ z = r\sin\vphi, \ r \ge 0$; it is bounded by the straight line $OB$ and has the angle of inclination $\vphi$ with respect to the half-plane of reference $z=0,\ y \ge 0$; see Fig.~\ref{fig body2}.

   \begin{figure}[h]
\begin{picture}(0,180)
\scalebox{0.95}{
\rput(3.8,0.45){
\psecurve[linecolor=brown](0.5,2)(3,0.1)(6.5,0.3)(9.5,2.5)(5,5)(0.5,2)(3,0.1)(6.5,0.3)
 \psline(-3,2)(0.5,2)
\psline[arrows=->,arrowscale=1.6](-3,2)(-4,0.25)
\psline[arrows=->,arrowscale=1.6](-3,2)(-3,6)
\psline[linewidth=0.6pt,linestyle=dashed](0.5,2)(9.4,2)
         \psline[arrows=->,arrowscale=1.6](9.4,2)(10.5,2)
\pscurve[linecolor=brown](0.5,2)(1.5,2.9)(3,3.6)(5.7,3.9)(7.9,3.3)(9.4,2)
\psline[linecolor=blue,linewidth=0.4pt](-3,2)(2.7,4.45)
               \psline[linecolor=blue,linewidth=0.4pt](-3,2)(3.1,3.64)
       \psline[linecolor=blue,linewidth=0.4pt](-3,2)(2.5,0.25)
\pscurve[linecolor=brown,linewidth=0.6pt](2.7,4.45)(2.8,4.4)(3.1,3.64)
(2.95,1)(2.6,0.29)(2.5,0.25)
\pscurve[linecolor=brown,linewidth=0.4pt,linestyle=dashed](2.7,4.45)(2.6,4.33)(2.4,3.64)
   (2.25,1)(2.4,0.35)(2.5,0.25)
   \psarc[linecolor=blue,linewidth=0.4pt](-3,2){0.25}{117}{240}
\rput(-3.4,2){\scalebox{0.85}{$\vphi$}}
\rput(-2.9,1.7){$O$}
\rput(9.6,1.7){$B'$}
\rput(0.34,1.8){\scalebox{0.9}{$B$}}
\psecurve[linecolor=brown,linewidth=0.6pt](0.85,3.15)(1.3,3.57)(1.67,3.03)(1.05,0.98)(0.6,1.1)
\psecurve[linecolor=brown,linewidth=0.4pt,linestyle=dashed](1.5,3.75)(1.3,3.57)(0.7,2.2)(1.05,0.98)(1.25,0.8)
                     \psline[linecolor=blue,linewidth=0.4pt](-0.2,2)(1.67,3.03)
                     \psline[linecolor=blue,linewidth=0.4pt](-0.2,2)(1.05,0.98)
                     \psline[linecolor=blue,linewidth=0.4pt](-0.2,2)(1.3,3.57)
\rput(-0.32,1.75){\scalebox{0.9}{$B_t$}}
\rput(3.45,3.4){\scalebox{0.9}{$M_\vphi$}}
\rput(1.9,2.8){\scalebox{0.8}{$M_{t,\vphi}$}}
       \rput(5.4,2.8){\scalebox{1.3}{$C_\vphi$}}
       \rput(-1.25,1.1){$\scalebox{1.3}{K}$}
           \rput(0.2,1.3){\scalebox{1.1}{$K_t$}}
                      \rput(-1,4.4){\scalebox{1.1}{$\Pi_\vphi$}}
              \rput(6,-0.5){$\scalebox{1.6}{C}$}
              \pspolygon[linewidth=0.2pt](-3,2)(10,2)(9,4)(-4,4)
              \psline[linewidth=0.2pt,arrows=->,arrowscale=2](-3,2)(-4,4)
                            \rput(-4.1,3.7){$r$}
              \rput(10.5,1.75){$x$}
              \rput(-3.5,0.5){$y$}
              \rput(-3.25,5.8){$z$}
   \psdots[dotsize=2.5pt](-3,2)(0.5,2)(-0.2,2) (3.1,3.64)(1.67,3.03)
   }}
\end{picture}
\caption{The convex body $C$ and some additional notation.}
\label{fig body2}
\end{figure}
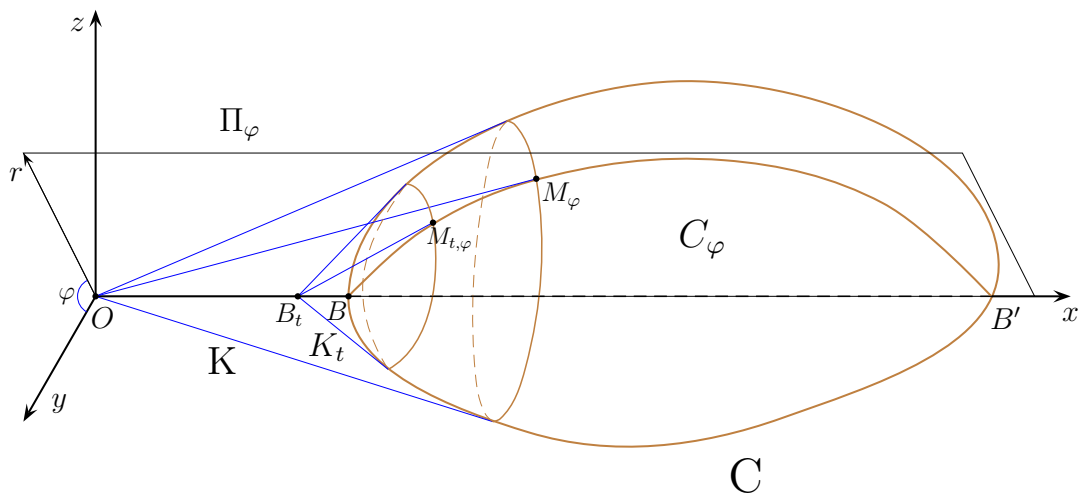  

The intersection of $C$ with $\Pi_\vphi$ is a 2-dimensional convex body $C_\vphi$ bounded below by a segment of the positive $x$-semiaxis and above by the concave curve $\pl C \cap \Pi_\vphi$. Draw the ray of support from $O$ to this curve and denote by $M_\vphi$ the closest to $O$ point of intersection of the ray with the curve. Denote by $\pl_- C_\vphi$ the part of the curve between $B$ and $M_\vphi$ (including $B$ and excluding $M_\vphi$). $\pl_- C_\vphi$ is the intersection of $\Pi_\vphi$ with $\pl_- C$, and it is the graph of a convex function $r \mapsto x(r,\vphi)$. It can be expressed as
$$
\pl_- C_\vphi = \pl C_\vphi \cap \text{int} \big( \text{Conv} (C_\vphi \cup \{O\}) \big) = \pl_- C \cap \Pi_\vphi.
$$
Denote by $\pl_+ C_\vphi$ the resting part of the curve between $M_\vphi$ and $B'$. It can be expressed as  
$$
\pl_+ C_\vphi = \pl C_\vphi \cap \pl \big( \text{Conv} (C_\vphi \cup \{O\}) \big) \setminus [B,\, B') = \pl_+ C \cap \Pi_\vphi.
$$


Draw the ray of support in $\Pi_\vphi$ from a point $B_t$ to the curve $\pl C \cap \Pi_\vphi$. The intersection of the ray with the curve is a line segment (which may degenerate to a point). Denote by $M_{t,\vphi}$ the endpoint of this segment closest to $B_t$. In particular, we have $M_{0,\vphi} = M_{\vphi}$ and $M_{1,\vphi} = B$.

Recall that condition (i) of Theorem \ref{t0} states that for any $\vphi$ the ray of support from $O$ to the curve in the half-plane $\Pi_\vphi$ touches this curve.

We also consider the following
\vspace{2mm}

{\bf Condition (ii)$'$.} {\it For any $0 < \tau < 1$ there exists $k = k(\tau)$ such that for any $\vphi$, the part of the curve $\pl_- C_\vphi$ between $M_{\tau,\vphi}$ and $M_{\vphi}$ is of class $C^1$ and has curvature $\le k$.}
\vspace{2mm}

We will prove the following theorem.
\vspace{2mm}

{\bf Theorem 1$'$.} {\it Let conditions (i) and (ii)$'$ be satisfied; then the statement of Theorem \ref{t0} holds true.}
\vspace{2mm}

According to condition (ii)$'$, the function $r \mapsto x(r,\vphi)$ is of class $C^1$.

Introduce some additional functions associated with the curve $\pl_- C_\vphi$.  At each point $(x(r,\vphi), r)$ draw the tangent line to $\pl_- C_\vphi$ in the plane containing the half-plane $\Pi_\vphi$. This line intersects the $x$-axis at a point $B_t$; thus, the  monotone non-increasing function $r \mapsto t(r,\vphi)$ is defined.

This function can also be found from the implicit equation
$$
\al(t) = x(r,\vphi) - r\, \frac{\pl x}{\pl r}(r,\vphi).
$$
Since the right hand side of this equation is continuous in $r$ and the function $\al$ is continuous and strictly increasing, we conclude that the function $r \mapsto t(r,\vphi)$ is continuous.

It is also convenient to use the generalized inverse to the function $t(\cdot, \vphi)$. Namely, for $0 \le t \le 1$ designate by $r(t, \vphi)$ the smallest value of $r$ such that $t(r,\vphi) = t$. The function $t \mapsto r(t,\vphi)$ defined this way is strictly decreasing and right semicontinuous. Additionally, we have $r(1,\vphi)= 0$ and $M_{t,\vphi} = (x(r(t,\vphi), \vphi),\, r(t,\vphi))$.

Further, we denote by $\s = \s(t, \vphi)$ the inverse slope of the corresponding tangent line. In other words, $\s(t, \vphi)$ is the cotangent of the angle of inclination of the tangent line to the curve $\pl_- C_\vphi$ in the half-plane $\Pi_\vphi$. For all $\vphi$, the function of one variable $t \mapsto \s(t, \vphi)$ is monotone non-increasing.

\begin{lemma}\label{l sigma}
(a) The function $\s(t, \vphi)$ is continuous.

(b) The partial derivative $\frac{\pl\s}{\pl\vphi}(t,\vphi)$ exists almost everywhere.
\end{lemma}

\begin{proof}
Fix $t$ and cross the cone $K_t$ by the plane perpendicular to the $x$-axis at the distance 1 from the vertex $B_t$ of the cone. The crossing plane is given by $x = \al(t) + 1$. The intersection of $K_t$ with the plane is a convex 2-dimensional body bounded, in the polar coordinates $r,\, \vphi$, by the curve $r = 1/\s(t,\vphi)$. Indeed, the ray from $B_t$ passing through a point $(r,\vphi)$ on the curve, has the slope $r$ with respect to the $x$-ray. On the other hand, it is known that the slope is $1/\s(t,\vphi).$

The function $r = 1/\s(t,\vphi)$ that determines the boundary of the convex 2-dimensional body is continuous in $\vphi$. Moreover, the curve is regular for almost all values of $\vphi$, and at exactly these values the function $\vphi \mapsto \s(t,\vphi)$ admits the derivative $\frac{\pl\s}{\pl\vphi}(t,\vphi)$. It follows that $\s(t,\vphi)$ is continuous in $\vphi$, and $\frac{\pl\s}{\pl\vphi}(t,\vphi)$ exists for almost all values $(t,\vphi)$. Thus, (b) is proved.





Draw the tangent lines in $\Pi_\vphi$ to the curve $\pl_- C_\vphi$ at its endpoints $B$ and $M_\vphi$, and fix the corresponding angles of inclination. When the point of tangency passes all points of the curve from $B$ to $M_\vphi$, the angle of inclination of the corresponding tangent line takes all intermediate values. Correspondingly, the value $\s(t,\vphi)$ takes all values intermediate between $\s(0,\vphi)$ and $\s(1,\vphi)$. It follows that the function $\s(t,\vphi)$ is continuous in $t$.

Now having that $\s$ is continuous in each of the variables $\vphi$ and $t$ and monotone in $t$, it is not difficult to conclude that $\s$ is continuous in both variables. Thus, (a) is also proved.
\end{proof}

\begin{lemma}\label{l curvature}
For any $0 < \tau < 1$ there exists $\gam = \gam(\tau) > 0$ such that for all $\vphi$ and $0 \le t_1 < t_2 \le \tau$,
\beq\label{Lip}
r(t_2,\vphi) - r(t_1,\vphi) \le -\gam\, (t_2 - t_1).
\eeq
\end{lemma}

\begin{proof}
The continuous function $\s(t,\vphi)$ is defined on the compact set $[0,\, 1] \times [0,\, 2\pi]$, and therefore, $|\s(t,\vphi)| \le c$ for a certain constant $c.$ It follows that for any $\vphi$ the cosine of the angle of inclination, with respect to the $r$-axis, of the curve $\pl_- C_\vphi$ in $\Pi_\vphi$ is greater than or equal to $1/\sqrt{1 + c^2}$.

Take $0 \le t_1 < t_2 \le \tau$ and introduce the shorthand notation $r_1 := r(t_1,\vphi), \ r_2 := r(t_2,\vphi), \ \psi_1 := \text{arccot}(\s(t_1,\vphi)), \ \psi_2 := \text{arccot}(\s(t_2,\vphi))$. We have $t_1 < t_2$, therefore $r_1 > r_2$.  The length of the part of curve between the points $M_{t_1,\vphi} = (x(r_1), r_1)$ and $M_{t_2,\vphi} = (x(r_2), r_2)$ does not exceed $\sqrt{1 + c^2} (r_1 - r_2)$, hence the angle between the tangent lines at these points does not exceed $k\sqrt{1 + c^2} (r_1 - r_2)$, that is,
$$
0 \le \psi_2 - \psi_1 \le k\sqrt{1 + c^2} (r_1 - r_2).
$$

Draw the line through the point $(x(r_1), r_1)$ parallel to the tangent line at $(x(r_2), r_2)$, and let $\bar\al$ be the $x$-coordinate of intersection of this line with the $x$-axis; see Fig.~\ref{fig curv}. Since the curve is concave, we have $\bar\al \ge \al(t_2)$.

     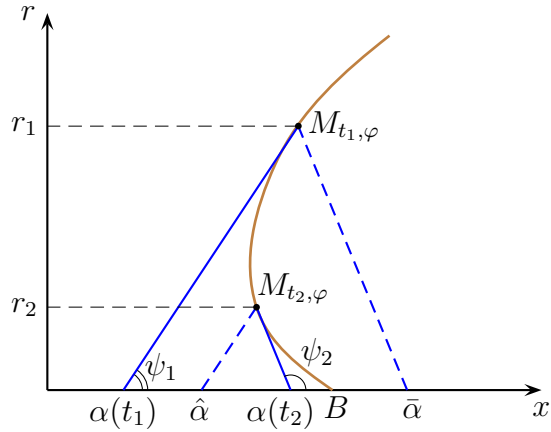
\begin{figure}[h]
\begin{picture}(0,170)
\scalebox{1}{
\rput(3.6,0.75){
\psline[arrows=->,arrowscale=1.5](0,0)(6.5,0)
\rput(6.5,-0.25){$x$}
\psline[arrows=->,arrowscale=1.5](0,0)(0,5)
\rput(-0.25,5){$r$}
\psecurve[linecolor=brown,linewidth=1pt](5,-0.5)(3.75,0)(2.75,1.1)(3,3)(4.5,4.7)(6,5.5)
\psline[linecolor=blue](1,0)(3.3,3.5)
            \psarc[linewidth=0.4pt](1,0){0.25}{0}{55}
                        \psarc[linewidth=0.4pt](1,0){0.32}{0}{55}
            \rput(1.5,0.3){$\psi_1$}
                       \rput(3.9,3.5){$M_{t_1,\vphi}$}
\psline[linecolor=blue](3.2,0)(2.75,1.1)
           \psarc[linewidth=0.4pt](3.2,0){0.2}{0}{115}
           \rput(3.55,0.5){$\psi_2$}
                      \rput(3.23,1.35){$M_{t_2,\vphi}$}
                      \rput(3.8,-0.225){$B$}
\psline[linecolor=blue,linestyle=dashed](4.732,0)(3.3,3.5)
   \psline[linecolor=blue,linestyle=dashed](2.027,0)(2.75,1.1)
   \rput(2,-0.25){$\hat\al$}
\rput(1,-0.3){$\al(t_1)$}
\rput(3.07,-0.3){$\al(t_2)$}
\rput(4.8,-0.25){$\bar\al$}
\psline[linewidth=0.3pt,linestyle=dashed](2.75,1.1)(0,1.1)
\rput(-0.3,1.1){$r_2$}
\psline[linewidth=0.3pt,linestyle=dashed](3.3,3.5)(0,3.5)
\rput(-0.3,3.5){$r_1$}
\psdots[dotsize=2.5pt](3.3,3.5)(2.75,1.1)
}}
\end{picture}
\caption{Auxiliary construction to Lemmae \ref{l curvature} and \ref{l conditions}.}
\label{fig curv}
\end{figure} 

Thus, we obtain
$$
\al(t_2) - \al(t_1) \le \bar\al - \al(t_1) = r_1\, \big( \s(t_1,\vphi) - \s(t_2,\vphi) \big) = r_1 (\cot\psi_1 - \cot\psi_2).
$$
Further, taking into account that the points $M_{t_1,\vphi} = (x(r_1), r_1)$ and $B = (\al(1), 0)$ are contained in $C$, we have
$$
r_1 \le \text{dist}\big( (x(r_1), r_1),\, (\al(1), 0) \big) \le \text{diam}(C),
$$
where diam$(C)$ means the diameter of $C$, and for some intermediate value $\psi_0 \in (\psi_1,\, \psi_2)$,
$$
\cot\psi_1 - \cot\psi_2 = \cot' \psi_0\, (\psi_1 - \psi_2) = (1 + \cot^2 \psi_0)\,  (\psi_2 - \psi_1) \le (1 + c^2)\, (\psi_2 - \psi_1).
$$

Finally, we have
$$
t_2 - t_1 \le \frac{\al(t_2) - \al(t_1)}{\min_{t \in [0,\tau]} \al'(t)} \le \frac{r_1 (\cot\psi_1 - \cot\psi_2)}{\min_{t \in [0,\tau]} \al'(t)}
$$
$$
\le \frac{\text{diam}(C)\, (1 + c^2)\, (\psi_2 - \psi_1)}{\min_{t \in [0,\tau]} \al'(t)} \le \frac{\text{diam}(C)\, (1 + c^2)\, k\sqrt{1 + c^2}}{\min_{t \in [0,\tau]} \al'(t)}\, \big(r_1 - r_2\big).
$$
It follows that the statement of the lemma holds true for
$$
\gam = \frac{\min_{t \in [0,\tau]} \al'(t)}{\text{diam}(C)\, k\, (1 + c^2)^{3/2}}.
$$
\end{proof}

Using Lemma \ref{l curvature}, one easily obtains the following statement.

\begin{corollary}\label{cor l curv}
Consider the map $g : (t, \vphi) \mapsto (r(t, \vphi), \vphi)$ from $[0,\, 1] \times [0,\, 2\pi)$ to $\RR^2$.
If a planar set $A$ has Lebesgue measure 0, then $g^{-1}(A)$ also has Lebesgue measure 0.
\end{corollary}

\begin{proof}
Let $0 < \tau < 1$, and take a Borel set $A \subset g([0,\, \tau] \times [0,\, 2\pi))$. Using that $|r(t_2, \vphi) - r(t_1, \vphi)| \ge \gam(\tau) |t_2 - t_1|$ for all $\vphi$ and $t_1,\, t_2 \in [0,\, \tau]$, one concludes that $|g^{-1}(A)| \le \frac{1}{\gam(\tau)}\, |A|$. In particular, if $|A| = 0$ then $|g^{-1}(A)| = 0$.

Let now $A \subset \RR^2$ have measure 0. Then for all $0 < \tau < 1$ the set $g^{-1} \big( A \cap g([0,\, \tau] \times [0,\, 2\pi)) \big)$ also has measure 0. Taking $0 < \tau_k < 1$,\, $k \in \mathbb{N}$ and $\tau_k \to 1$ as $k \to \infty$, one has
$$
g^{-1}(A) = \cup_{k\in\mathbb{N}} g^{-1} \big( A \cap g([0,\, \tau_k] \times [0,\, 2\pi)) \big) \cup \big( \{1\} \times [0,\, 2\pi) \big);
$$
thus, $g^{-1}(A)$ is the union of countably many sets of Lebesgue measure 0, and therefore, $|g^{-1}(A)| = 0$.
\end{proof}

Denote by $K_t(\vphi)$ the 2-dimensional cone in $\Pi_\vphi$ bounded below by the $x$-axis and above by the ray of support through $B_t$; see Fig.~\ref{fig cone}. Clearly, $K_t(\vphi) = K_t \cap \Pi_\vphi$, and $K_t(\vphi)$ is circumscribed about $C_\vphi$. In the $(x,r)$-coordinate system one has
$$
K_t(\vphi) = \{ (x,r) : r \ge 0,\ x \ge \s(t,\vphi) r + \al(t) \}.
$$
The body $C_\vphi$ can be represented in terms of the 2-dimensional circumscribed cones as
$$
C_\vphi = \big( \cap_{0 \le t \le 1} K_t(\vphi) \big) \cap \text{Conv} \big( C_\vphi,\, \{O\} \big).
$$

  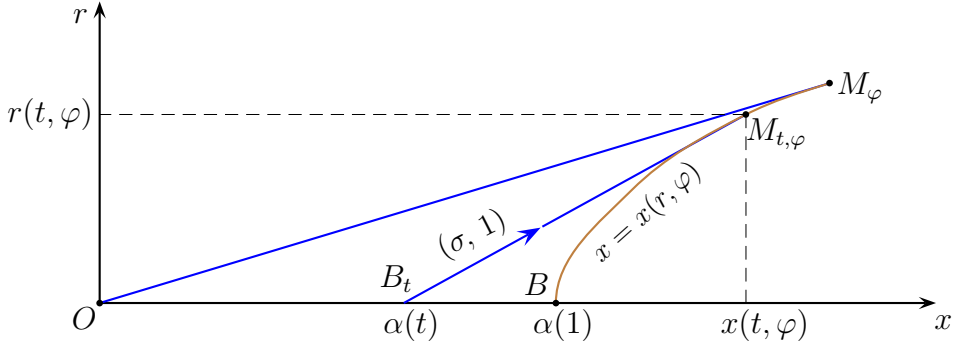
\begin{figure}[h]
\begin{picture}(0,150)
\scalebox{1}{
\rput(7.25,0.75){
\psline[arrows=->,arrowscale=1.5](-6,0)(5,0)
\psline[arrows=->,arrowscale=1.5](-6,0)(-6,4)
           \psline[linecolor=blue](-6,0)(3.6,2.915)
    \psline[linecolor=blue,arrows=->,arrowscale=2](-2,0)(-0.2,1)
    \psline[linecolor=blue](-0.1775,1.0125)(2.5,2.5)
\psline[linewidth=0.3pt,linestyle=dashed](2.5,2.5)(-6,2.5)
\psline[linewidth=0.3pt,linestyle=dashed](2.5,2.5)(2.5,0)
\rput(5.1,-0.25){$x$}
\rput(-6.25,3.8){$r$}
\rput(-1.9,-0.3){$\al(t)$}
\rput(0.1,-0.3){$\al(1)$}
\rput(-6.2,-0.2){$O$}
\rput(-2.1,0.32){$B_t$}
\rput(-0.25,0.25){$B$}
\rput(-6.65,2.5){$r(t,\vphi)$}
\rput(2.75,-0.3){$x(t,\vphi)$}
\rput{27}(-1.1,0.9){\scalebox{0.95}{$(\s,\, 1)$}}
\rput{42}(1.2,1.2){\scalebox{0.9}{$x=x(r,\vphi)$}}
\rput(4,2.85){$M_\vphi$}
\rput(2.9,2.25){$M_{t,\vphi}$}
\psecurve[linecolor=brown](1,-1.5)(0,0)(1,1.5)(2.5,2.5)(3.6,2.915)(4.5,3.15)
\psdots[dotsize=2.5pt](0,0)(-6,0)(2.5,2.5)(3.6,2.915)
}}
\end{picture}
\caption{The cone $K_t(\vphi)$ is bounded below by the $x$-axis and above by the ray $B_t M_{t,\vphi}$.}
\label{fig cone}
\end{figure} 

Using the definitions of the functions $\al,\, x,\, t,\, \s$, one readily obtains the formulas
$$
x(r,\vphi) = \s(t(r,\vphi),\vphi)\, r + \al(t(r,\vphi)) \qquad \text{and} \qquad \frac{\pl x}{\pl r}(r,\vphi) = \s(t(r,\vphi),\vphi).
$$
In terms of the inverse function $t \mapsto r(t,\vphi)$ we have
\beq\label{formulas}
x(r(t,\vphi),\vphi) = \s(t,\vphi)\, r(t,\vphi) + \al(t), \qquad \frac{\pl x}{\pl r}(r(t,\vphi),\vphi) = \s(t,\vphi).
\eeq
If all functions in \eqref{formulas} are continuously differentiable, one differentiates by $\vphi$ both parts in the former formula in \eqref{formulas} to obtain
$\frac{\pl x}{\pl r}\, \frac{\pl r}{\pl\vphi} + \frac{\pl x}{\pl\vphi} = \s\, \frac{\pl r}{\pl\vphi} + r\, \frac{\pl\s}{\pl\vphi},$
and using the latter formula in \eqref{formulas} one gets
\beq\label{key formula}
\frac{1}{r(t,\vphi)}\, \frac{\pl x}{\pl\vphi}(r(t,\vphi),\vphi) = \frac{\pl\s}{\pl\vphi}(t,\vphi).
\eeq
Since differentiability is not assumed a priori, we need some additional effort to justify this formula. We shall prove that \eqref{key formula} holds for almost all $(t,\vphi) \in [0,\, 1] \times [0,\, 2\pi)$.

For the sake of brevity, in the next paragraph we introduce the shorthand notation $r = r(t,\vphi)$,\, $x = x(r,\vphi)$,\, $\Del x = x(r,\vphi + \Del\vphi) - x(r,\vphi)$, and $\Del\s = \s(r,\vphi + \Del\vphi) - \s(r,\vphi)$ for $\Del\vphi \in \RR$.

We have
\beq\label{f1}
\s(t,\vphi) = \frac{x - \al(t)}{r}.
\eeq
On the other hand, the tangent line from the point $(\al(t), 0)$ to the graph of the function $r \mapsto x(r, \vphi+\Del\vphi)$ (which is a concave curve) lies above the ray from $(\al(t), 0)$ through $(x + \Del x,\, r)$; see Fig.~\ref{fig formula}.
     \begin{figure}[h]
\begin{picture}(0,165)
\scalebox{1}{
\rput(2.6,0.5){
\psline[arrows=->,arrowscale=1.5](0,0)(8,0)
\psline[linewidth=0.4pt,linestyle=dashed](0,4)(8,4)
\rput(5,-0.25){$x$}
\rput(6.7,-0.25){$x+\Del x$}
\rput(1,-0.3){$\al(t)$}
\psline[linewidth=0.2pt,linestyle=dashed](5,4)(5,0)
\psline[linewidth=0.2pt,linestyle=dashed](6.5,4)(6.5,0)
\psline[arrows=->,arrowscale=1.5](0,0)(0,5)
\rput(-0.25,4){$r$}
\psecurve[linecolor=brown](3.1,-1)(3,0)(3.3,1.75)(4.1,3)(5,4)(5.65,4.5)(6.5,5)
\psecurve[linecolor=brown](2.7,-1)(3,0)(3.5,1)(4.5,2.5)(5.6,3.5)(6.5,4)(7.5,4.4)(8.5,4.75)
\psline[linecolor=blue](1,0)(6,5)
\psline[linecolor=blue](1,0)(7.6,4.8)
\psdots[dotsize=2.5pt](5,4)(6.5,4)
\rput{45}(4.55,3.15){\scalebox{0.875}{$x=x(r,\vphi)$}}
\rput{57}(4,1.3){\scalebox{0.8}{$x=x(r,\vphi+\Del\vphi)$}}
}}
\end{picture}
\caption{Auxiliary construction to the proof of formula \eqref{key formula}.}
\label{fig formula}
\end{figure}
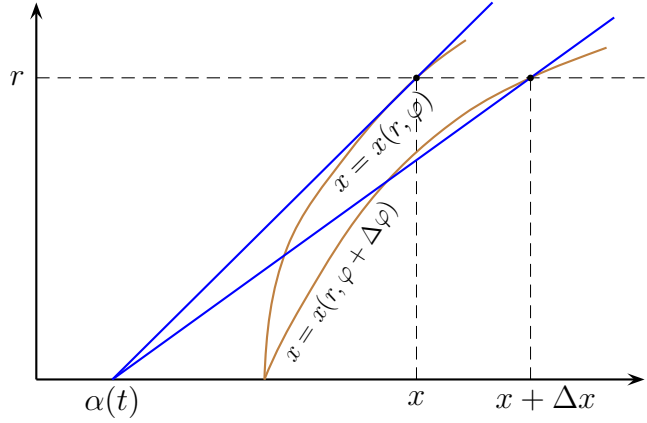 
It follows that
\beq\label{f2}
\s(t,\vphi+\Del\vphi) \le \frac{x + \Del x - \al(t)}{r}.
\eeq
From \eqref{f1} and \eqref{f2} one derives that
$$
\frac{\Del\s}{\Del\vphi} \le \frac{1}{r}\, \frac{\Del x}{\Del\vphi} \ \ \text{ for } \ \Del\vphi > 0 \qquad \text{and} \qquad \frac{\Del\s}{\Del\vphi} \ge \frac{1}{r}\, \frac{\Del x}{\Del\vphi} \ \ \text{ for } \ \Del\vphi < 0.
$$
Going to the limit $\Del\vphi \to 0$, one obtains formula \eqref{key formula}, provided that the partial derivatives in both sides of \eqref{key formula} exist.

The statement (b) of Lemma \ref{l sigma} guarantees that the right hand side of \eqref{key formula} exists for almost all $(t,\vphi) \in [0,\, 1] \times [0,\, 2\pi)$.

Note that $\pl_- C$ is the graph of a convex function $x$ of variables $y,\,z$. In polar coordinates $y = r\cos\vphi,\, z = r\sin\vphi$ this function takes the form $(r,\, \vphi) \mapsto x(r,\vphi).$ For almost all $(r,\vphi)$ the corresponding point of $\pl_- C$ is regular, and therefore, the partial derivative $\frac{\pl x}{\pl\vphi}(r,\vphi)$ exists. In other words, the set $A = \{ (r,\vphi) : \,\frac{\pl x}{\pl\vphi}(r,\vphi) \ \text{does not exist} \}$ has Lebesgue measure 0, and using Corollary \ref{cor l curv} one concludes that the set $g^{-1}(A) = \{ (t,\vphi) : \,\frac{\pl x}{\pl\vphi}(r(t,\vphi),\vphi) \ \text{does not exist} \}$ also has Lebesgue measure 0. Thus, the left hand side of \eqref{key formula} also exists for almost all $(t,\vphi)$.

\begin{lemma}\label{l ineq}
For all $\vphi$ and all $t_1$,\, $t_2$,
\beq\label{ineq0}
r(t_2,\vphi) \big[ \s(t_1,\vphi) - \s(t_2,\vphi) \big] \le \al(t_2) - \al(t_1) \le r(t_1,\vphi) \big[ \s(t_1,\vphi) - \s(t_2,\vphi) \big].
\eeq
\end{lemma}

\begin{proof}
The curve $\pl_- C_\vphi$ lies below all its lines of support. The lines of support through the points $M_{t_1,\vphi} = (x(t_1,\vphi), r(t_1,\vphi))$ and $M_{t_2,\vphi} = (x(t_2,\vphi), r(t_2,\vphi))$ are, respectively, given by the equations $x = \s(t_1,\vphi)\, r + \al(t_1)$ and $x = \s(t_2,\vphi)\, r + \al(t_2)$ . The point $M_{t_1,\vphi}$ lies to the right of the support line through $M_{t_2,\vphi}$, hence
$$
x(t_1,\vphi) \ge \s(t_2,\vphi)\, r(t_1,\vphi) + \al(t_2),
$$
and therefore,
\beq\label{ineq1}
\s(t_1,\vphi)\, r(t_1,\vphi) + \al(t_1) \ge \s(t_2,\vphi)\, r(t_1,\vphi) + \al(t_2).
\eeq

Similarly, the point $M_{t_2,\vphi}$ lies to the right of the support line through $M_{t_1,\vphi}$, hence
$$
x(t_2,\vphi) \ge \s(t_1,\vphi)\, r(t_2,\vphi) + \al(t_1),
$$
and therefore,
\beq\label{ineq2}
\s(t_2,\vphi)\, r(t_2,\vphi) + \al(t_2) \ge \s(t_1,\vphi)\, r(t_2,\vphi) + \al(t_1).
\eeq

From \eqref{ineq1} and \eqref{ineq2} one obtains the statement of Lemma \ref{l ineq}.
\end{proof}

A plane of support to $C$ through a point $\xi \in \pl C$ divides the space into the closed half-space containing $C$ and the open half-space disjoint with $C$. If $\xi \in \pl_- C$ then $O$ is contained in the half-space disjoint with $C$, and therefore, for a regular point $\xi$,\, $\langle n_\xi,\, \overrightarrow{OP} \rangle < 0$ for any point $P \in C$. Hence this inequality remains valid for all points $P \in K \setminus \{ O \}$.

On the other hand,  if $\xi \in \pl_+ C$ then $O$ is contained in the half-space containing $C$, and therefore, for a regular point $\xi$,\, $\langle n_\xi,\, \overrightarrow{O\xi} \rangle \ge 0$. Taking $P = \xi$, one concludes that the inequality $\langle n_\xi,\, \overrightarrow{OP} \rangle \ge 0$ is valid for a point $P \in K \setminus \{ O \}$.

Denote
\beq\label{Om-}
\Om_- = \Om_-(K) = \big\{ n \in S^2 : \langle n,\, \overrightarrow{OP} \rangle < 0 \ \, \text{for all} \ P \in K \setminus \{ O \} \big\}
\eeq
and
\beq\label{Om+}
\Om_+ = \Om_+(K) = \big\{ n \in S^2 : \langle n,\, \overrightarrow{OP}  \rangle \ge 0 \ \, \text{for some} \ P \in K \setminus \{ O \} \big\}.
\eeq
We have $S^2 = \Om_- \cup \Om_+.$ It follows from the aforementioned argument that
$$
\NNN_C^{-1}(\Om_-) = \pl'_- C \qquad \text{and} \qquad \NNN_C^{-1}(\Om_+) = \pl'_+ C.
$$


Designate by $[0,\, 1] \times [0,\, 2\pi)'$ the set of values $(t, \vphi) \in [0,\, 1] \times [0,\, 2\pi)$ that correspond to regular points of $\pl C$. According to Corollary \ref{cor l curv}, $[0,\, 1] \times [0,\, 2\pi)'$ is a full measure subset of $[0,\, 1] \times [0,\, 2\pi)$.

 Define the mapping $H_C : [0,\, 1] \times [0,\, 2\pi)' \to \Om_-$ by $H_C(t,\vphi) = n_{M_{t,\vphi}}$.

The intersection of the curve $\pl_- C_\vphi$ with the ray of support from $B_t$ is a line segment with an endpoint at $M_{t,\vphi}$. A plane of support to $C$ at a point of this segment contains the ray of support, and therefore, is also a plane of support to any other point of the segment. It follows that the points of the segment are either all regular, or singular. In the case of regularity, the normal vectors at the points of the segment coincide with $n_{M_{t,\vphi}}$.

\begin{zam}\label{z normal}
Note that

(a) $(t,\vphi) \in [0,\, 1] \times [0,\, 2\pi)'$ iff the points on the ray $\pl K_t \cap \Pi_\vphi$ are regular points of the cone $K_t$, and

(b) the vector $H_C(t,\vphi)$ coincides with the outer normal to the cone $K_t$ at any point of the ray $\pl K_t \cap \Pi_\vphi$.
\end{zam}

\begin{lemma}\label{l integral}
For any Borel set $D \subset \Om_-$,
\beq\label{integr_form}
\nu_C(D) = \int \!\!\! \int_{H_C^{-1}(D)} \sqrt{1 + \s^2(t,\vphi) + \Big(\frac{\pl\s}{\pl\vphi}(t,\vphi)\Big)^2}\ \frac{1}{2}\, d\big[ -r^2(t,\vphi) \big]  d\vphi.
\eeq
\end{lemma}

The expression in the right hand side of \eqref{integr_form} should be understood as a repeated integral, where the interior integral is a Stieltjes integral with the (monotone increasing) generating function $t \mapsto -r^2(t,\vphi)$.

\begin{proof}
Notice that in general, if a measurable subset $\OOO \subset \pl C$ of the surface of a convex body projects injectively on the $yz$-plane then, denoting by pr$_{yz} \OOO$ the projection and by $\theta(y,z)$ the angle of inclination, with respect to the $x$-axis, of the point on $\pl C$ that projects to $(y,z)$, the area of $\OOO$ can be expressed as
$$
|\OOO| = \int\!\!\!\int_{\text{pr}_{yz}\OOO} \frac{1}{|\cos\theta(y,z)|}\, dy\, dz.
$$

The surface $\pl_- C$ can be parameterized by $(r,\vphi) \mapsto  (x(r,\vphi),\, r\cos\vphi,\, r\sin\vphi)$, where $r,\, \vphi$ are the polar coordinates on the $yz$-plane. It is convex, hence the perpendicular vector to the surface
 $$
 \Big( r; \, \ \sin\vphi\, \frac{\pl x}{\pl\vphi} - r\cos\vphi\, \frac{\pl x}{\pl r}; \ \, -r\sin\vphi\, \frac{\pl x}{\pl r} - \cos\vphi\, \frac{\pl x}{\pl\vphi} \Big).
$$
 exists almost everywhere and is a measurable function of $(r,\vphi)$.

The cosine of the angle of inclination of this vector with respect to the $x$-axis is $r/\sqrt{r^2 + \big( \frac{\pl x}{\pl\vphi} \big)^2 + r^2 \big( \frac{\pl x}{\pl r} \big)^2}$, and its inverse is
$\sqrt{1 + \big( \frac{\pl x}{\pl r} \big)^2 + \frac{1}{r^2}\, \big( \frac{\pl x}{\pl\vphi} \big)^2 }$. We know that $\frac{\pl x}{\pl r} = \s$, and by formula \eqref{key formula}, almost everywhere $\frac{1}{r}\, \frac{\pl x}{\pl\vphi} = \frac{\pl\s}{\pl\vphi}$, therefore the inverse cosine equals $\sqrt{1 + \s^2 + (\frac{\pl\s}{\pl\vphi})^2}$.
Thus, we obtain the formula in polar coordinates
$$
\nu_C(D) = \int\!\!\!\int \sqrt{1 + \s^2(t(r,\vphi),\vphi) + \Big(\frac{\pl\s}{\pl\vphi}(t(r,\vphi),\vphi)\Big)^2}\ r\, dr\, d\vphi,
$$
where the integration is taken over the set $\{ (r,\vphi) : \, H_C(t(r,\vphi), \vphi) \in D \}$.

Making the change of variables induced by the map $g(t,\vphi) = (r(t,\vphi),\vphi)$, one obtains formula \eqref{integr_form}.
\end{proof}

Take a positive $C^1$ function $\eta(t),\ t \in [0,\, 1]$ with $\eta(0) = 1$, and introduce the auxiliary functions
$$
\tilde r(t,\vphi) = \eta(t) r(t,\vphi)
$$
and
\beq\label{tilde al}
\tilde\al(t) = \int_0^t \eta(\xi)\, d\al(\xi).
\eeq

One obviously has $\tilde r(0,\vphi) = r(0,\vphi)$ and $\tilde r(1,\vphi) = 0$.

\begin{opr}\label{o admissible}
The function $\eta$ is called {\it admissible}, if

(a) for any $0 < \tau < 1$ there exists $\tilde\gam = \tilde\gam(\tau) > 0$ such that
$$\tilde r(t_2,\vphi) - \tilde r(t_1,\vphi) \le -\tilde\gam\, (t_2 - t_1)$$
for all $0 \le t_1 < t_2 \le \tau$ and for all $\vphi$;

(b) $\tilde\al(1) < |OB'|$.
\end{opr}

The following lemma guarantees that there is a huge variety of admissible functions.

\begin{lemma}\label{l eta admissible}
For all $0 < T < 1$ there exists $c = c(T) > 0$ such that any function of the form $\eta(t) = e^{\chi(t)}$ with $\chi$ of class $C^1$,\, $\chi(0) = 0, \ |\chi'(t)| \le c$ for $0 \le t \le T$ and $\chi(t) = \chi(T)$ for $T < t \le 1$ is admissible.
\end{lemma}

\begin{proof}
Recall that by Lemma \ref{l curvature}, for any $0 < \tau < 1$ there exists $\gam = \gam(\tau) > 0$ such that for all $\vphi$ and $0 \le t_1 < t_2 \le \tau$, \eqref{Lip} is satisfied.

Fix $T$ and take positive $c$ satisfying the inequalities $c \le \gam(T)/(2\,\text{diam}(C))$ and $e^{cT} < |OB'| / |OB|$.

For each $\vphi$ the function $t \mapsto r(t,\vphi)$ is monotone decreasing and therefore is differentiable for almost all $t$. At each point $t \in [0,\, T]$ of differentiability we have $\frac{\pl r}{\pl t}(t,\vphi) \le -\gam(T)$.

Let $\eta$ satisfy the conditions of the lemma. If $r$ is differentiable in $t$ at a certain $t$, then the product $\eta r$ also is, and    
%
%
we use the estimates ${\eta'(t)}/{\eta(t)} = \chi'(t) \le c$ and ${\frac{\pl r}{\pl t}(t,\vphi)}/{r(t,\vphi)} \le -{\gam(T)}/{\text{diam}(C)} \le -2c$ to obtain
$$
\frac{\frac{\pl}{\pl t}(\eta(t) r(t,\vphi))}{\eta(t) r(t,\vphi)}  = \frac{\eta'(t)}{\eta(t)} + \frac{\frac{\pl r}{\pl t}(t,\vphi)}{r(t,\vphi)} \le -c.
$$

Of course the function $t \mapsto \eta(t) r(t,\vphi)$ has bounded variation. Additionally, we have just obtained that at each point of differentiability of this function, for $t \in [0,\, T]$,\, ${\frac{\pl}{\pl t}(\eta(t) r(t,\vphi))}/(\eta(t) r(t,\vphi)) \le -c.$

Take $0 \le t_1 < t_2 \le T$.  
We use Lebesgue's decomposition to represent $\eta r$ as the sum of an absolutely continuous function, a singular function, and a jump function. The set of singular values $t$, that is, values at which the derivative of the singular function is not defined and values corresponding to jumps, has measure zero. We first take $t_1$ and $t_2$ not in this set.

Take $\ve > 0$ and consider a covering of the set of singular values $t$ by the union of countably many disjoint open intervals, $J = \cup_i (a_i,\, b_i)$, with the total sum of lengths being smaller than $\ve$, that is, $|J| = \sum_i (b_i - a_i) < \ve$. One can choose $J$ in such a way that it does not contain $t_1$ and $t_2$.

Substitute the function $\eta(\cdot) r(\cdot, \vphi)$ with the auxiliary continuous function $\hat\eta(\cdot)$ which coincides with $\eta(\cdot) r(\cdot, \vphi)$ on $[0,\, 1] \setminus J$ and is affine on each interval $(a_i,\, b_i)$. One easily checks that this function is absolutely continuous and  its derivative on $[0,\, 1] \setminus J$ coincides with the derivative of $\eta(\cdot) r(\cdot, \vphi)$. Hence,
$$
\ln\big(\eta(t_2) r(t_2,\vphi)\big) - \ln\big(\eta(t_1) r(t_1,\vphi)\big) = \ln\big(\hat\eta(t_2)\big) - \ln\big(\hat\eta(t_1)\big) = \int_{t_1}^{t_2} \frac{\hat\eta'(t)}{\hat\eta(t)}\, dt
$$
\beq\label{eq eta}
= \int_{[t_1, t_2] \setminus J} \frac{\frac{\pl}{\pl t}(\eta(t) r(t,\vphi))}{\eta(t) r(t,\vphi)}\, dt + \sum_i \big[ \ln(\eta(b_i) r(b_i,\vphi)) - \ln(\eta(a_i) r(a_i,\vphi)) \big].
\eeq
The integral in \eqref{eq eta} is less than $-c(t_2 - t_1 -\ve)$, and we use the inequalities $r(b_i,\vphi)) < r(a_i,\vphi))$ to estimate the sum in the right hand side of \eqref{eq eta} as follows:
$$
\sum_i \big[ \ln(\eta(b_i) r(b_i,\vphi)) - \ln(\eta(a_i) r(a_i,\vphi)) \big] < \sum_i \big[ \ln(\eta(b_i) - \ln(\eta(a_i) \big]
$$
$$
\le \max_t \frac{\eta'(t)}{\eta(t)}\, \sum_i (b_i - a_i) \le \ve\, c.
$$

Thus, one has $\ln\tilde r(t_2,\vphi) - \ln\tilde r(t_1,\vphi)) = \ln\big(\eta(t_2) r(t_2,\vphi)\big) - \ln\big(\eta(t_1) r(t_1,\vphi)\big)  \le -c(t_2 - t_1) + 2\ve c$, and taking into account that $\ve > 0$ is arbitrary, one obtains $\ln\tilde r(t_2,\vphi) - \ln\tilde r(t_1,\vphi)) \le -c(t_2 - t_1)$, and so,
$$
\tilde r(t_2,\vphi) - \tilde r(t_1,\vphi) \le -\big( e^{c(t_2 - t_1)} - 1 \big) \tilde r(t_2,\vphi) \le - c(t_2 - t_1)\, e^{-cT} \min_\vphi r(T,\vphi).
$$

Now assume that one or both values $t_1$,\, $t_2$ are singular. Since each singular value is a limiting point of nonsingular ones, and using continuity in $t$ of the function $\eta(t) r(t,\vphi)$, by the limiting process we come to the same inequality $\tilde r(t_2,\vphi) - \tilde r(t_1,\vphi) \le - [c\, e^{-cT} \min_\vphi r(T,\vphi)]\, (t_2 - t_1)$.

It remains to note that in the case $T < \tau$, for two values $T \le t_1 < t_2 \le \tau$ we have
$$
\tilde r(t_2,\vphi) - \tilde r(t_1,\vphi) = \eta(T) \big( r(t_2,\vphi) - r(t_1,\vphi) \big) \le -e^{-cT} \gam(\tau)\, (t_2 - t_1).
$$
It follows that condition (a) in the definition of admissibility is verified with the constant $\tilde\gam(\tau) = \min \big\{ c\, e^{-cT} \min_\vphi r(T,\vphi), \ e^{-cT} \gam(\tau) \big\}$.

Further, using that $\al(1) = |OB|$ and $\max_\xi \eta(\xi) \le e^{cT}$, one has
$$
\tilde\al(1) = \int_0^1 \eta(\xi)\, d\al(\xi) \le \max_\xi \eta(\xi)\, \int_0^1 d\al(\xi) \le e^{cT} \cdot \al(1) = e^{cT} |OB| < |OB'|.
$$
Thus, condition (b) is also satisfied.
\end{proof}

In what follows we assume that the function $\eta$ is admissible.

\begin{lemma}\label{l ineq2}
For any $t_1$ and $t_2$ the following formula analogous to \eqref{ineq0} holds
$$
\tilde r(t_2,\vphi) \big( \s(t_1,\vphi) - \s(t_2,\vphi) \big) \le \tilde\al(t_2) - \tilde\al(t_1) \le \tilde r(t_1,\vphi) \big( \s(t_1,\vphi) - \s(t_2,\vphi) \big).
$$
\end{lemma}

\begin{proof}
We shall give the proof for the case $t_1 < t_2$; for $t_2 < t_1$ the argument is completely similar.

After suitable substitutions this formula is transformed into
\beq\label{ineq with tilde}
\eta(t_2) r(t_2,\vphi) \big( \s(t_1,\vphi) - \s(t_2,\vphi) \big) \le \int_{t_1}^{t_2} \eta(\xi)\, d\al(\xi) \le \eta(t_1) r(t_1,\vphi) \big( \s(t_1,\vphi) - \s(t_2,\vphi) \big).
\eeq
Consider partitions of $[t_1,\, t_2]$ into subintervals, $\PPP = \{ t_1 = t^{(0)} < t^{(1)} < \ldots < t^{(n-1)} < t^{(n)} = t_2 \}$, and take the limit as $\del(\PPP) := \max_{1 \le i \le n} (t^{(i)} - t^{(i-1)}) \to 0.$ In the formulas below we make use of the left inequality in \eqref{ineq0} and take into account that the function $\eta(t)r(t,\vphi)$ is monotone decreasing in $t.$
$$
\int_{t_1}^{t_2} \eta(\xi)\, d\al(\xi) = \lim_{\del(\PPP)\to 0} \sum_{i=1}^n \eta(t^{(i)}) \big( \al(t^{(i)}) - \al(t^{(i-1)}) \big)
$$
$$
\ge \lim_{\del(\PPP)\to 0} \sum_{i=1}^n \eta(t^{(i)}) r(t^{(i)},\vphi) \big( \s(t^{(i-1)},\vphi) - \s(t^{(i)},\vphi) \big)
$$
$$
\ge \eta(t_2) r(t_2,\vphi) \sum_{i=1}^n \big( \s(t^{(i-1)},\vphi) - \s(t^{(i)},\vphi) \big)  = \eta(t_2) r(t_2,\vphi) \big( \s(t_1,\vphi) - \s(t_2,\vphi) \big).
$$
The left inequality in \eqref{ineq with tilde} is proved.

Now we make use of the right inequality in \eqref{ineq0}.
$$
\int_{t_1}^{t_2} \eta(\xi)\, d\al(\xi) = \lim_{\del(\PPP)\to 0} \sum_{i=1}^n \eta(t^{(i-1)}) \big( \al(t^{(i)}) - \al(t^{(i-1)}) \big)
$$
$$
\le \lim_{\del(\PPP)\to 0} \sum_{i=1}^n \eta(t^{(i-1)}) r(t^{(i-1)},\vphi) \big( \s(t^{(i-1)},\vphi) - \s(t^{(i)},\vphi) \big)
$$
$$
\le \eta(t_1) r(t_1,\vphi) \sum_{i=1}^n \big( \s(t^{(i-1)},\vphi) - \s(t^{(i)},\vphi) \big) = \eta(t_1) r(t_1,\vphi) \big( \s(t_1,\vphi) - \s(t_2,\vphi) \big).
$$
The right inequality in \eqref{ineq with tilde} is proved.
\end{proof}

Fix $\eta$, for any $\vphi$ take the function $t \mapsto \tilde r(t,\vphi) = \eta(t) r(t,\vphi)$, and define the generalized inverse function $r \mapsto \tilde t(r,\vphi)$,\, $r \in [0,\, r(1,\vphi)]$ by the relation $\tilde t(r,\vphi) := \inf\{ t : \tilde r(t,\vphi) \le r \}$. The generalized inverse function $\tilde t(\cdot,\vphi)$ is continuous and monotone non-increasing. The intervals where it is constant correspond to jumps of the function $\eta(\cdot) r(\cdot, \vphi)$.

For each $\vphi$ and $t \in [0,\, 1]$ consider the planar cone $\tilde K_t(\vphi)$ in the half-plane $\Pi_\vphi$  given by the inequalities
\beq\label{line}
x \ge \s(t,\vphi)\, r + \tilde\al(t), \qquad r \ge 0.
\eeq
The cone $\tilde K_t(\vphi)$ is bounded below by the $x$-axis and above by the ray with the origin $\tilde B_t := (\tilde\al(t), \, 0)$  through the point $\tilde M_{t,\vphi} := (\s(t,\vphi)\, \tilde r(t,\vphi) + \tilde\al(t),\, \tilde r(t,\vphi))$, and its vertex is at $\tilde B_t$. Note that for $t = 0$ one has $\tilde M_{0,\vphi} = (\s(0,\vphi)\, r(0,\vphi), \ r(0,\vphi) ) = M_\vphi.$

The union over all $t$ of these planar cones is the 3-dimensional convex cone $\tilde K_t := \cup_\vphi \tilde K_t(\vphi)$ with the vertex at $\tilde B_t$, which is actually the translation of $K_t$ along the $x$-axis by $\tilde\al(t) - \al (t) = \int_0^t (\eta(\tau) - 1)\, d\al(\tau)$. One has, in particular, $\tilde B_0 = O$,\, $\tilde K := \tilde K_0 = K$, and $\tilde K_0(\vphi) = K_0(\vphi).$ Denote $\tilde B := \tilde B_1$. Condition (b) in Definition \ref{o admissible} means that $\tilde B$, and therefore all the points $\tilde B_t$,\, $0 \le t \le 1,$ lie to the left of $B'$.

Define the new planar convex body $\tilde C_\vphi$ by
$$
\tilde C_\vphi = (\cap_{0 \le t \le 1} \tilde K_t(\vphi)) \cap \rm{Conv}(C_\vphi, \{O\})
$$
and the corresponding 3-dimensional convex body $\tilde C$ by
\beqo\label{tilde C}
\tilde C = \cup_\vphi \tilde C_\vphi = \Big( \cap_{0 \le t \le 1} \tilde K_t \Big) \cap {\rm Conv}(C, \{O\}).
\eeqo
Note that it does not follow automatically from these definitions that all cones $\tilde K_t(\vphi)$ are circumscribed about $\tilde C_\vphi$ or that the 3-dimensional cones $\tilde K_t$ are circumscribed about $\tilde C$. This is justified in the following Lemma \ref{l4}.

$\tilde C_\vphi$ is bounded below by a segment of the positive $x$-semiaxis and above by a concave curve with the endpoints at $\tilde B$ and $B'.$ This curve can also be defined as $\tilde C_\vphi \cap \Pi_\vphi$.

Similarly to what was done above for $C_\vphi$, we draw the ray of support from $O$ in the half-plane $\Pi_\vphi$ to this curve, denote by $\pl_- \tilde C_\vphi$ the part of the curve between $\tilde B$ and the closest to $O$ point of intersection of the ray with the curve (excluding this point), and denote by $\pl_+ \tilde C_\vphi$ the complementary part of the curve between this point of intersection and $B'$. From each point $\tilde B_t$,\, $0 \le t \le 1$, draw the ray of support to this curve and denote by $\tilde\s_{t,\vphi}$ the inverse slope of the ray and by $H_{\tilde C}(t,\vphi)$ the outer normal to $\tilde C$ at the closest to $\tilde B_t$ point of intersection of this ray with the curve.

\begin{lemma}\label{l4}
(a) The ray of support in $\Pi_\vphi$ from $O$ to the curve $\tilde C_\vphi \cap \Pi_\vphi$ is $x = \s(0,\vphi)\, r, \ r \ge 0$. The closest to $O$ point of intersection of the ray with the curve is $M_\vphi$, and
$$
\pl_+ \tilde C_\vphi = \pl_+ C_\vphi.
$$

(b) The ray of support from $\tilde B_t$ to the curve is $x = \s(t,\vphi)\, r + \tilde\al(t), \ r \ge 0$,  and therefore,
\beq\label{sig}
\tilde\s_{t,\vphi} = \s_{t,\vphi}.
\eeq
Additionally, $\tilde M_{t,\vphi}$ is the closest to $\tilde B_t$ point of intersection of the ray with the curve.

(c) At each point $(t,\vphi) \in [0,\, 1] \times [0,\, 2\pi)'$,\, $H_{\tilde C}(t,\vphi)$ exists and
$$
H_{\tilde C}(t,\vphi) = H_{C}(t,\vphi).
$$

(d) ${\rm Conv}(\tilde C_\vphi, \{O\}) = {\rm Conv}(C_\vphi, \{O\})$.
\end{lemma}

\begin{proof}
Fix $\vphi$. One easily derives from Lemma \ref{l ineq2} that each point $\tilde M_{t,\vphi} = (\s(t,\vphi)\, \tilde r(t,\vphi) + \tilde\al(t),\, \tilde r(t,\vphi))$ lies to the right of all rays of the kind $x = \s(t',\vphi)\, r + \tilde\al(t'),\, r \ge 0$\, $(0 \le t' \le 1)$ and therefore belongs to the intersection of the cones $\cap_{0 \le t \le 1} \tilde K_t(\vphi)$.

Further, for any $0 \le t \le 1$, both the points $B'$ and $M_\vphi$ lie to the right of the ray $x = \s(t,\vphi)\, r + \tilde\al(t), \ r \ge 0$ containing $\tilde M_{t,\vphi}$, and $\tilde M_{t,\vphi}$ lies to the right of the line $OM_\vphi$, and therefore, $\tilde M_{t,\vphi}$ belongs to  the triangle $OM_\vphi B'$. Since this triangle is contained in $\text{Conv} (C_\vphi, \{O\})$, we have
$$
\tilde M_{t,\vphi} \in \text{Conv} (C_\vphi, \{O\}).
$$

It follows that the point $\tilde M_{t,\vphi} $ lies in $\tilde C_\vphi = \big( \cap_{0 \le t \le 1} \tilde K_t(\vphi) \big) \cap \big( \text{Conv} (C_\vphi, O) \big)$. Since it also lays on the ray $x = \s(t,\vphi)\, r + \tilde\al(t), \ r \ge 0$, we conclude that this ray is the ray of support from $\tilde B_t$ to $\tilde C_\vphi$, and therefore, $\tilde\s_{t,\vphi} = \s_{t,\vphi}.$ Thus, formula \eqref{sig} in (b) is proved.

In particular, the ray $x = \s(0,\vphi)\, r, \ r \ge 0$ is the ray of support from $O$ to the curve $\tilde C_\vphi \cap \Pi_\vphi$, and the point $\tilde M_{0,\vphi} = M_\vphi$ lies in the intersection of this ray with the curve. Further, the $r$-coordinate of the closest to $O$ point of the intersection is smaller than or equal to the $r$-coordinate of $M_\vphi$ and  greater than or equal to the $r$-coordinate of $\tilde M_{t,\vphi}$ for all $t > 0$ (these coordinates are equal, correspondingly, to $r(0,\vphi)$ and $\tilde r(t,\vphi)$). Since the function $t \mapsto \tilde r(t,\vphi)$ is right semicontinuous, we have $\lim_{t\to 0^+} \tilde r(t,\vphi) = \tilde r(0,\vphi) = r(0,\vphi)$, and therefore, the closest to $O$ point is $M_\vphi$.

Exactly the same argument can be used to prove that $\tilde M_{t,\vphi}$ is the closest to $\tilde B_t$ point of intersection of the ray $x = \s(t,\vphi)\, r + \tilde\al(t), \ r \ge 0$ with the curve. Thus, the proof of (b) is completed.

The curve $\pl_+ C_\vphi$ lies in the intersection of cones $\cap_{0 \le t \le 1} \tilde K_t(\vphi)$, and therefore, in $\tilde C_\vphi = (\cap_{0 \le t \le 1} \tilde K_t(\vphi)) \cap \rm{Conv}(C_\vphi, \{O\})$. Since $\pl_+ C_\vphi$ lies in the boundary of $\rm{Conv}(C_\vphi, \{O\})$, it also lies in the boundary of $\tilde C_\vphi$, and additionally, is bounded by the points $M_\vphi$ and $B'$. Since $\pl_+ \tilde C_\vphi$ is also bounded by $M_\vphi$ and $B'$, we have $\pl_+ \tilde C_\vphi = \pl_+ C_\vphi$, and (a) is proved.

Remark \ref{z normal} is applicable also to $\tilde C$, that is, (a) $H_{\tilde C}(t,\vphi)$ exists {\it iff} the points of the ray $\pl\tilde K_t \cap \Pi_\vphi$ are regular points of the cone $\tilde K_t$, and (b) $H_{\tilde C}(t,\vphi)$ coincides with the outer normal to $\tilde K_t$ at any point of the ray $\pl\tilde K_t \cap \Pi_\vphi$. Since $\tilde K_t$ is a translation of $K_t$ along the $x$-ray and $\Pi_\vphi$ is invariant with respect to this translation, the points of the ray $\pl \tilde K_t \cap \Pi_\vphi$ are regular points of $\tilde K_t$ {\it iff} the points of the ray $\pl K_t \cap \Pi_\vphi$ are regular points of $K_t$, and in the case of regularity, the outer normal to $\tilde K_t$ at a point of $\pl \tilde K_t \cap \Pi_\vphi$ coincides with the outer normal to $K_t$ at a point of $\pl K_t \cap \Pi_\vphi$. It follows that $H_{\tilde C}(t,\vphi)$ is defined on $[0,\, 1] \times [0,\, 2\pi)'$ and coincides with $H_{C}(t,\vphi)$, and (c) is proved.

Note that the set Conv$(C_\vphi, \{O\})$ is bounded by the line segments $OB'$ and $OM_\vphi$ and by the curve $\pl_+ C_\vphi$, and therefore, coincides with Conv$(\pl_+ C_\vphi, \{O\})$. Using that $\pl_+ C_\vphi \subset \tilde C_\vphi$, we obtain
$$
\text{Conv}(C_\vphi, \{O\}) = \text{Conv}(\pl_+ C_\vphi, \{O\}) \subset \text{Conv}(\tilde C_\vphi, \{O\}).
$$
On the other hand, since $\tilde C_\vphi \subset \rm{Conv}(C_\vphi, \{O\})$, we have
$$
\text{Conv}(\tilde C_\vphi, \{O\}) \subset \text{Conv}\big(\text{Conv}(C_\vphi, \{O\}), \{O\} \big) = \text{Conv}(C_\vphi, \{O\}).
$$
Thus, (d) is proved.
\end{proof}

 The statement (b) of the following lemma \ref{l conditions} is, in a sense, inverse to Lemma \ref{l curvature}.

\begin{lemma}\label{l conditions}
(a) The ray $x = \s(0,\vphi)\, r, \ r \ge 0$ is tangent to $\tilde C_\vphi$.

(b) For any $0 < \tau < 1$ there exists $\tilde k = \tilde k(\tau)$ such that for any $\vphi$, the part of the curve $\pl_- \tilde C_\vphi$ between the points $\tilde M_{t,\vphi}$ and $M_\vphi$ is of class $C^1$ and has curvature $\le \tilde k$.
\end{lemma}

\begin{proof}
The proof of (b) is essentially an inversion of the argument of Lemma \ref{l curvature}. We use the same Fig.~\ref{fig curv} with a slight correction of notation: $M_{t_i,\vphi}$ and $\al(t_i)$,\, $i = 1,\,2$ in the figure should be replaced, respectively, by $\tilde M_{t_i,\vphi}$ and $\tilde\al(t_i)$.

Take $0 \le t_1 < t_2 \le \tau$. We use the shorthand notation $r_1 := \tilde r(t_1,\vphi), \ r_2 := \tilde r(t_2,\vphi), \ \psi_1 := \text{arccot}(\s(t_1,\vphi)), \ \psi_2 := \text{arccot}(\s(t_2,\vphi))$. One has $t_1 < t_2$, therefore $r_1 > r_2$ and $\psi_1 < \psi_2$.

The length of the part of the curve between the points $\tilde M_{t_1,\vphi}$ and $\tilde M_{t_2,\vphi}$ (to be referred to as {\it the curve} later on in the proof) is greater than or equal to $r_1 - r_2$, and since the function $\tilde r$ is admissible, for $\tilde\gam = \tilde\gam(\tau)$ one has $r_1 - r_2 \ge \tilde\gam (t_2 - t_1)$.

Draw the line through the point $\tilde M_{t_2,\vphi}$ parallel to the tangent line at $\tilde M_{t_2,\vphi}$, and let $\hat\al$ be the $x$-coordinate of intersection of this line with the $x$-axis. Since the curve is concave, we have $\hat\al \ge \tilde\al(t_1)$. Thus, we have
$$
\tilde\al(t_2) - \tilde\al(t_1) \ge \tilde\al(t_2) - \hat\al = r_2\, \big( \s(t_1,\vphi) - \s(t_2,\vphi) \big).
$$

Note that $\s(t_1,\vphi) - \s(t_2,\vphi) = \cot\psi_1 - \cot\psi_2 > \psi_2 - \psi_1$. Introducing the value $\tilde R_\tau := \min_\vphi \tilde r(\tau,\vphi)$ and using that $r_2 \ge \tilde R_\tau$, we get
$$
\text{\big(\!\! length of the curve\!\! \big)} \ge r_1 - r_2 \ge \tilde\gam (t_2 - t_1) \ge \tilde\gam\, \frac{\tilde\al(t_2) - \tilde\al(t_1)}{\max_{t \in [0,1]} \al'(t)} \ge \frac{1}{\tilde k}\, (\psi_2 - \psi_1),
$$
where
$$
\tilde k = \tilde k(\tau) = \frac{\max_{t \in [0,1]} \al'(t)}{\tilde\gam\, \tilde R_\tau}.
$$
Thus, the angle of inclination of the support line is a continuous function of the arc length. This implies that the curve is of class $C^1$. Claim (b) of the lemma is proved.

Since the functions $\s(\cdot, \vphi)$ and $\tilde t(\cdot, \vphi)$ are continuous, the slope $1/\s(\tilde t(r, \vphi), \vphi)$ of the curve $\pl_- \tilde C_\vphi$ converges to $1/\s(0,\vphi)$ when the point of the curve approaches $M_\vphi$. It follows that the ray $x = \s(0,\vphi)\, r, \ r \ge 0$ is tangent to $\pl_- \tilde C_\vphi$ at the endpoint $M_\vphi$. On the other hand, according to condition (i) of Theorem \ref{t0}$'$, this ray is also tangent to $\pl_+ C_\vphi$ at the endpoint $M_\vphi$. Taking into account that $\pl_+ C_\vphi = \pl_+ \tilde C_\vphi$, one concludes that the ray is tangent to $\tilde C_\vphi$. Thus, claim (a) of the lemma is also proved.
\end{proof}

Denote $\pl_+ \tilde C = \cup_\vphi \pl_+ \tilde C_\vphi = \pl\tilde C \cap \pl\big( \text{Conv} (C \cup \{O\}) \big)$. It follows from claim (d) of Lemma \ref{l4} that $\pl_+ \tilde C = \pl_+ C$, and therefore,
\beq\label{pl+}
\pl_+ C \subset \pl\tilde C.
\eeq

Since $\tilde K$ coincides with $K$, the sets $\Om_-(K)$ and $\Om_+(K)$, defined by \eqref{Om-} and \eqref{Om+}, coincide with $\Om_-(\tilde K)$ and $\Om_+(\tilde K)$, respectively. It follows that $\NNN_{\tilde C}^{-1}(\Om_+) = \pl'_+ \tilde C = \pl'_+ C$ and $\NNN_{\tilde C}^{-1}(\Om_+) = \pl'_+ \tilde C$ and $\NNN_{\tilde C}\big\rfloor_{\pl'_+ \tilde C} = \NNN_{C}\big\rfloor_{\pl'_+ C}$.

Claims (b), (c) of Lemma \ref{l4} and Lemma \ref{l conditions} mean that the body $\tilde C$ satisfies the conditions of Theorem \ref{t0}$'$, the corresponding auxiliary functions $\tilde\s(t,\vphi)$ and $H_{\tilde C}(t,\vphi)$ coincide with $\s(t,\vphi)$ and $H_{C}(t,\vphi)$, respectively, and the function $r(t,\vphi)$ should be repaced with $\tilde r(t,\vphi) = \eta(t) r(t,\vphi)$. Therefore Lemma \ref{l integral} is also applicable to the body $\tilde C$.

For any Borel set $D \subset \Om_+$ we have
\beq\label{area+}
\nu_{\tilde C}(D) =  |\NNN_{\tilde C}^{-1}(D)| =  |\NNN_{C}^{-1}(D)| = \nu_{C}(D).
\eeq
For a Borel set $D \subset \Om_-$ we get
\beq\label{area-}
\nu_{\tilde C}(D)  = \int\!\!\!\int_{H_C^{-1}(D)} \sqrt{1 + \s^2(t,\vphi) + \Big(\frac{\pl\s}{\pl\vphi}(t,\vphi)\Big)^2}\ \frac{1}{2}\, d\big[ -\eta^2(t) r^2(t,\vphi) \big] d\vphi.
\eeq

Now prove Theorem \ref{t0}$'$. Fix $0 < T < 1$ and take $c = c(T)$ as explained in Lemma \ref{l eta admissible}. Choose a $C^1$ function $\theta : [0,\, 1] \to \RR$ such that (i) $\theta(t) = 0$ for $t \in [T,\, 1]$, (ii) $|\theta(t)| < 1$ for all $t$, and (iii) $\big| \frac{d}{dt} \big[ \frac 12 \ln(1 \pm \theta(t)) \big] \big| \le c(T)$ for all $0 \le t \le T$. It is straightforward to check that for any $s \in [-1,\, 1]$ the function $\eta(\cdot, s) = \sqrt{1 + s\theta(\cdot)}$ satisfies the conditions of Lemma \ref{l eta admissible}, and therefore is admissible.

Consider the family of functions $\eta(\cdot, s) = \sqrt{1 + s\, \theta(\cdot)}$,\, $s \in [-1,\, 1]$ and the corresponding family of convex bodies $C(s) = C^{(\theta)}(s)$. One obviously has $C(0) = C$; thus, statement (a) of the theorem is true.

Formula \eqref{pl+} remains valid when $\tilde C$ is replaced with $C(s)$; thus, statement (b) of the theorem is also true.

According to formulas \eqref{area+} and \eqref{area-}, for a Borel set $D \subset \Om_+$ we have $\nu_{C(s)}(D) = \nu_{C}(D)$, and for a Borel set $D \subset \Om_-$,
 $$
\nu_{C(s)}(D)  = \int \!\!\!\int_{H_C^{-1}(D)} \sqrt{1 + \s^2(t,\vphi) + \Big(\frac{\pl\s}{\pl\vphi}(t,\vphi)\Big)^2}\, \frac{1}{2}\, d\big[ -(1 + s\, \theta(t)) r^2(t,\vphi) \big] d\vphi.
$$
Using these relations, one easily derives that for any Borel set $D$,
$$
\nu_{C(s)}(D) = \nu_C(D) + s\, \big( \nu_{C(1)}(D) - \nu_{C}(D) \big),
$$
and so, statement (c) of the theorem is true.

Further, the signed measure that serves as a director vector, $\nu_{C^{(\theta)}(1)} - \nu_C$, satisfies the relations $(\nu_{C^{(\theta)}(1)} - \nu_C)(D)  = 0$, if $D \subset \Om_+$, and
 $$
(\nu_{C^{(\theta)}(1)} - \nu_C)(D)  = \int \!\!\!\int_{H_C^{-1}(D)} \sqrt{1 + \s^2(t,\vphi) + \Big(\frac{\pl\s}{\pl\vphi}(t,\vphi)\Big)^2}\, \frac{1}{2}\, d\big[ -\theta(t)\, r^2(t,\vphi) \big]  d\vphi,
$$
if $D \subset \Om_-$.

Let $\Theta$ be the set of $C^1$ functions $\theta$ satisfying (i), (ii), and (iii). The union of the 1-dimensional subspaces spanned by all the director vectors corresponding to $\theta \in \Theta$ contains all signed measures induced by $C^1$ function $\theta$ equal to 0 on $[T,\, 1]$, that is, the measures $\nu = \nu^{(\theta)}$ defined by
$$
\nu^{(\theta)}(D)  = \left\{
\begin{array}{ll}
\ \ 0, & \text{if } \ D \subset \Om_+;\\
\int \!\!\int_{H_C^{-1}(D)} \, \frac{1}{2}\sqrt{1 + \s^2(t,\vphi) + \Big(\frac{\pl\s}{\pl\vphi}(t,\vphi)\Big)^2} \times & \\
\times d\Big[ -\theta(t)\, r^2(t,\vphi) \Big]  d\vphi, & \text{if } \ D \subset \Om_-.
\end{array}
\right.
$$
All such measures form an infinite-dimensional linear space. Thus, statement (d) of the theorem is proved.

\section*{Acknowledgements}

This work was supported by Portuguese funds through CIDMA - Center for Research and Development in Mathematics and Applications, and FCT - Portuguese Foundation for Science and Technology, within the project UID/MAT/04106/2013.

\end{document}